\documentclass[11pt, twoside, a4paper]{article}
\usepackage[english]{babel}
\usepackage{amsmath,amssymb,amsthm,epsfig}

\newtheorem{lem}{Lemma}[section]
\newtheorem{defi}{Definition}[section]
\newtheorem{rem}{Remark}[section]

\newtheorem{thm}{Theorem}[section]
\newtheorem{theorem}{Theorem}[section]
\newtheorem{lemma}[theorem]{Lemma}
\newtheorem{e-proposition}[theorem]{Proposition}

\newtheorem{e-definition}[theorem]{Definition\rm}

\newtheorem{example}{Example}
\newtheorem{cor}{Corollary}[section]
\newtheorem{pro}{Proposition}[section]

\def\T{{\mathbb T}}

\begin{document}


\title{Ergodic Transport Theory and Piecewise Analytic Subactions for Analytic Dynamics}

\author{A. O. Lopes \,\footnote{Instituto de Matem\'atica, UFRGS, 91509-900 - Porto Alegre, Brasil. Partially supported by CNPq, PRONEX -- Sistemas Din\^amicos, Instituto do Mil\^enio, and beneficiary of CAPES financial support }, E. R. Oliveira\footnote{Instituto de Matem\'atica, UFRGS, 91509-900 - Porto Alegre, Brasil} and D. Smania \,\footnote{Departamento de Matem\'atica, ICMC-USP 13560-970
S\~ao Carlos, Brasil. Partially supported by   CNPq 310964/2006-7,  and  FAPESP 2008/02841-4}}
\date{\today} 
\maketitle


\begin{abstract}

We consider a piecewise analytic  real expanding map $f: [0,1]\to [0,1]$
of degree $d$ which preserves orientation, and a real  analytic positive potential $g: [0,1] \to \mathbb{R}$.  We assume the map and the potential have a complex analytic extension to a neighborhood of the interval in the complex plane.  We also assume $\log g$ is well defined for this extension.

 It is known in Complex Dynamics that under the above hypothesis, for the given
potential  $\beta \,\log g$, where $\beta$ is a real
constant, there exists a real analytic
eigenfunction $\phi_\beta$ defined on $[0,1]$ (with a complex analytic extension) for the Ruelle operator of $\beta \,\log g$.


Under some assumptions we show that   $\frac{1}{\beta}\, \log \phi_\beta$ converges and  is a piecewise analytic calibrated subaction.

Our theory can be applied  when $\log g(x)=-\log f'(x)$. In that case we relate the involution kernel to the so called scaling function.

\end{abstract}


\bigskip

\maketitle

\newpage

\centerline{ 0. INTRODUCTION}
\bigskip

We consider a piecewise  real analytic  expanding map $f: [0,1]\to [0,1]$ of degree $d$ which preserves orientation
and a real analytic positive potential $g: [0,1] \to \mathbb{R}$.

 We assume the map and the potential have a complex analytic extension to a neighborhood of the interval in the complex plane.  We also assume that
 $\log g$ is well defined for this complex neighborhood and for the extension of $g$.

In our notation $A=\log g$, with $A$ analytic, then we
denote
$$ m(A)\,=\,\max_{\nu \text{\, an invariant probability for $f$}}
\int A(x) \; d\nu(x),$$ and $\mu_{\infty\, A}$ any probability which realizes the maximum value. Any one
of these probabilities $\mu_{\infty\, A}$ is called a maximizing
probability for $A$.
In general these probabilities do not necessarily give  positive weight to every open set.

\bigskip

An important result in Complex Dynamics is the following: under the above hypothesis, for a given real
analytic potential  $\beta \,\log g$, where $\beta\geq 0$ is a real
constant, there exists a real analytic positive
eigenfunction $\phi_\beta$ defined on $[0,1]$ for the real Ruelle operator $P_{\beta \log g}$  of the potential $\beta \,\log g$ (see \cite{PU} \cite{MS} \cite{Ru}).  The existence of a complex analytic extension of $g$ to the interval $[0,1]$ (see, for instance, section 2.5 beginning in page 96 \cite{Ba}, or, \cite{Ru}, \cite{MM}) is a key point in our proof.

We denote $\mu_\beta$ the equilibrium state for $\beta \log g$.
We recall that any accumulation point  $\mu_{\beta_n}, n \to \infty$,  is a maximizing probability for the real function $\log g$ (restricted to the interval $[0,1]$, see for instance  \cite{CG} \cite{Bousch1} \cite{CLT}. We will present precise definitions later.

It is known that any convergent subsequence of the equicontinuous family
$\frac{1}{\beta}\, \log \phi_\beta$ is a calibrated subaction (see \cite{CLT}).
Calibrated subactions play a very important role  in the understanding of the properties of the maximizing probabilities  (see \cite{Jenkinson1} \cite{CLT} \cite{BLT}).

A pertinent question is to know if  there exists a real analytic calibrated subaction?
There are examples where there is no real analytic calibrated subaction (see \cite{Bousch1}). Under what hypothesis one can find real analytic calibrated subactions? Is it possible to get piecewise real analytic calibrated subactions under some reasonable conditions? Our purpose here is to address these questions.

\bigskip

 A natural strategy would be to consider the complex extension of $\frac{1}{\beta}\, \log \phi_\beta$ to a certain complex neighborhood $O_\beta$ of $[0,1]$ and then to use the criteria of normal families when $\beta \to \infty$.
One  problem we have to face in this approach is that the results in the literature concerning the existence of the eigenfunction $\phi_\beta$ do not give a sharp information on the size of $O_\beta$ when $\beta$ changes. Our result shows that in general there is no uniform control of this size in the limit $\beta \to \infty$. This implies that the naive strategy has low chance to work.

\section{Definitions and statement of the main result}

A calibrated subaction  for $A=\log g$ is a  function $V$ such that

$$ \sup_{y\, \text{such that}\, f(y)\,=\,x}\,  \,\{\,V(y) \,+ \,\log g (y) \,-\, m ( \log g)\,\}\,= \, V( x)  . $$

If the maximizing probability is unique the calibrated subaction is unique, up to  an additive constant (see \cite{Bousch1} (Lemme C) or \cite{BLT} (Proposition 5)).
\bigskip

In Statistical Mechanics the parameter $\beta\geq 0$ is associated to the inverse of the temperature.
Then, one can say that the limit probability of $\mu_\beta$, when $\beta \to \infty$, corresponds to the case of the equilibrium at temperature zero (see \cite{BLT}, \cite{BLL}).
We refer the reader to \cite{Jenkinson1} \cite{CG} \cite{HY} \cite{Mo} \cite{GL1} \cite{Le} \cite{BQ} \cite{CH} and \cite{CLT} for general references and definitions on Ergodic Optimization.

We recall that the   Bernoulli  space is the set
$\{1,2,...,d\}^\mathbb{N}=\Sigma$. A general element $w$ in
$\Sigma$ is denoted by $ w=(w_0,w_1,..,w_n,..)$.

In section 6 we will assume that $d=2$.

We denote $\hat{\Sigma}$ the set $\Sigma \times [0,1]$ and $\psi_i$
indicates the $i$-th inverse branch of $f$. We also denote by
$\sigma$ the shift on $\Sigma$. Finally,  $\T^{-1}$ is the backward
shift on $\hat{\Sigma}$ given by $\T^{-1}(w,x) = ( \sigma (w),
\psi_{w_0} (x)).$ In order to analyze the analytic properties of the dynamics of $f$ we have to consider the underlying dynamics of the inverse branches, and, then it is natural to consider the extend system $\T$ acting on $\hat{\Sigma}$. This kind of approach (in some sense) appears also in the study of the scaling function (see \cite{PU}).

\begin{defi}
Consider $A:[0,1]\to \mathbb{R}$ H\"older.
We say that  $ W: \hat{\Sigma} \to \mathbb{R}$ is a involution kernel for $A$, if
there is a H\"older function  $A^*:\Sigma \to \mathbb{R}$ such
that
$$A^*(w)= A\circ \T^{-1}(w,x)+ W \circ \T^{-1}(w,x) -
W (w,x).$$

We say that $A^* $ is a dual potential of $A$, or, that $A$ and $A^*$ are in involution.
\end{defi}

Above we denote $A(x)$ and $A^*(w)$ to stress the difference of the domains of each one. Note that $A\circ \T^{-1}(w,x)= A(\psi_{w_0} (x)).$

\begin{rem} \label{re}
In order to show $W$ is an involution kernel for $A$ we just have to show that
$A\circ \T^{-1}(w,x)+ W \circ \T^{-1}(w,x) -
W (w,x)$ is continuous and just depends on $w$ (see \cite{BLT}).
\end{rem}

Given a  H\"older potential $A=\log g$, the existence and properties of an associated  H\"older continuous involution kernel $W$ was
presented in \cite{BLT}, for the purpose of getting a Large Deviation Principle.

We show here the existence of $W(w,x)$, $x \in [0,1], w \in
\{1,...,d\}^\mathbb{N}$, which is an analytic involution kernel for
$A(x)=\log g(x)$, and a relation with the dual potential
$A^*(w)=(\log g)^*(w)$ defined in the Bernoulli space
$\{1,...,d\}^\mathbb{N}$. In this case we have  $W:
\{1,...,d\}^\mathbb{N}\times [0,1] \to \mathbb{R}$, and, by analytic
we mean: for each $w \in \{1,...,d\}^\mathbb{N}$ fixed, the function
$W(w,.)$ has a complex analytic extension to a neighborhood of
$[0,1]$.
\bigskip


Here we assume that the maximizing probability for $A$ is unique which implies the maximizing probability for $A^*$ is also unique (see \cite{BLT}). We denote by $V^*$ the calibrated subaction for $A^*$.

We denote by $I^*$ the deviation function for $A^*$ (see \cite{BLT}).

Suppose $V$ is the limit of  a subsequence $\frac{1}{\beta_n}\, \log \phi_{\beta_n}$, where $\phi_{\beta_n}$ is an eigenfunction
of the Ruelle operator for $\beta_n A$. Suppose $V^*$ is obtained in an analogous way for $A^*$.  Then, there exists $\gamma$ such that
\begin{equation}\label{exp}
\gamma+ V(x ) =
\sup_{ w \in \Sigma}\, [\,W(w, x) - V^* (w)- I^*(w)\,].
\end{equation}

This expression has interesting relations with the additive eigenvalue problem (see \cite{Ba} \cite{CD})

  We consider on $\Sigma=\{0, ...,d-1\}^\mathbb{N}$ the lexicographic order. We will consider, by technical reasons, the case where $f:(0,1) \to (0,1)$ has positive derivative. In the most of the cases we will consider, $d=2$, in order to avoid an unnecessary heavy notation.


Following \cite{LOT} we define:

\begin{defi} \label{rom} We say a  continuous  $G: \hat{\Sigma}=\Sigma\times [0,1]
 \to \mathbb{R}$ satisfies the twist condition on
$ \hat{\Sigma}$, if  for any $(a,b)\in
\hat{\Sigma}=\Sigma\times[0,1] $ and $(a',b')\in\Sigma\times [0,1]
$, with $a'> a$, $b'>b$, we have
\begin{equation}
G(a,b) + G(a',b')  <  G(a,b') + G(a',b).
\end{equation}
\end{defi}

\begin{defi} We say a continuous $A: [0,1] \to \mathbb{R}$   satisfies the twist condition, if some  of its involution kernels
satisfies the twist condition.

\end{defi}

Note that if the above is true for some involution kernel it will be also true for any
involution kernel (see \cite{LOT}).

We will assume the twist condition for $W$ (sometimes called supermodular condition as in section 5.2 in \cite{Mi}), which is a  very
natural assumption for the cost in optimization problems (see
\cite{Ba} and the Monge condition in \cite{Del}).

The twist condition will assure that for the lexicographic order in
$\Sigma$ (can be any lexicographic order) the multi-valuated function $x \to w(x)$ is monotonous  decreasing (to be proved later). In
the case $f$ is a two to one map (that is, $d=2$),  a special point, which will be
called turning point, will play an important role.

The turning point $c$ (see fig. 1)  is defined by
$$c=\sup\{x\, |\, \, w(x)=(1\, w_1\, w_2 ...)\,\,\text{for some the possible }\,\,w(x)\}.$$

All results before section 6 are for the general case of a finite $d$. However, our main result, which is  Theorem \ref{maincor}, is for the case $d=2$. It claims that:

\begin{thm}

Assume that

a) the maximizing probability $\mu_{\infty\, A}$ is unique,

b)  has support in
a periodic orbit,

c)  $\log g$ is twist.

If  $d=2$ and the turning point $c$ is eventually periodic for $f$, then the calibrated sub-action $V:[0,1]\to
\mathbb{R}$ for the potential $A=\log g$ is piecewise analytic, with
a finite number of domains of analyticity.
\end{thm}

There are several  examples where the hypothesis of the theorem are true (see section 7). We show that expression (\ref{exp}) above can be used to find explicit calibrated subactions in some cases (see Example \ref{ex1} in section 7).

\bigskip
{\bf Motivation and discussion on assumptions}
\bigskip

As a motivation for the study of the above problem we mention the papers \cite{BKRU} \cite{T1}
which consider the fat attractor.  For a fixed potential $A=\log g$ (called $\tau$ in the notation of \cite{BKRU}) there exist an extra-parameter
$\lambda$. In \cite{BKRU} it is shown that the boundary of this attractor is related the graph of a certain function $u_\lambda$. When $\lambda \to 1$, we have that this $u_\lambda$ (normalized)
converges to a calibrated subaction for $A$ (see \cite{Bousch-walters} \cite{BCLMS}).
One of the conjectures presented in \cite{BKRU}, when translated to our language, claims that, if $A$ is $C^2$-generic, then the $u_\lambda$ is piecewise differentiable. The function denoted by $S$ in \cite{T1} corresponds to the involution kernel here. The techniques we consider here, namely, duality and the involution kernel, will be used on that context in a forthcoming paper in order to understand the unstable manifold of some special points in the boundary of the attractor.

In the setting of the fat attractor   \cite{T1} \cite{BKRU} the turning point corresponds to the projection on $S^1$ of the intersection
of certain unstable manifolds in the boundary of the attractor \cite{LO}.

\bigskip

The theory described here can be applied  when $\log g(x)=-\log f'(x)$. In that case we relate the involution kernel $W$ to the scaling function (see \cite{PU} \cite{GJQ} \cite{MM}). The dual potential $A^*$  of $A=-\log f'(x)$ will be the scaling function. The dual relation, via the involution kernel,  we consider here is a generalization of the relation of $-\log f'(x)$ and the scaling function. More precisely, in this case,   $e^{W(w,x)}$, $(w,x) \in \Sigma\times [0,1]$,   coincides with  the function $|D \psi_w
(x)|$ on the variables $(w,x)$ of \cite{MM}.

\bigskip

The twist condition (see  \cite{Gol}) on the involution kernel (it is a condition
that depends just on $A$) plays the same role in Ergodic Transport Theory than  the convexity hypothesis in Aubry-Mather Theory (see \cite{Mat} \cite{CI} \cite{Fa} \cite{Man}). Here we will assume this hypothesis which
was first considered   in \cite{LMST} and \cite{LOT}.
Examples of potentials $A$ such that the corresponding involution
kernel satisfies the twist condition appear there. The twist
condition is an open property in the variation of the analytic
potential $A=\log g$ defined in a fixed open complex neighborhood of
the interval $[0,1]$.

It will be clear from our proof that in the case the support of the maximizing probability is not a periodic orbit (a Cantor set for instance), then, one gets an infinite number of distinct  domains of analyticity. In this case the turning point will not be eventually periodic.

We point out that a main conjecture in Ergodic Optimization claims that generically (in the H\"older topology) on the potential $A$ the maximizing probability has support in a periodic orbit (see \cite{CLT} for related results). Therefore, the assumption that the maximizing probability is a periodic orbit makes sense.

We point out that in the case $f$ reverses orientation (like ,$-2 x $ (mod $1$)), then there is no potential $A=\log g$ which is twist for the dynamics on $\Sigma\times [0,1]$.  A careful analysis (for different types of Baker maps) of when it is possible for $A$ to be  twist for a given dynamics $f$ is presented in \cite{LOT}. We will not consider this case here.

\bigskip

{\bf Strategy of the proof}
\bigskip

 By compactness, for each  $x$  there exists {\bf  at least one}
$w(x)$ such that
$$\gamma+ V( x ) =
\sup_{ w \in \Sigma}\, [\,W(w, x) - V^* (w)-
I^*(w)\,]=
$$
$$
[\,W(w(x) , x ) - V^* (w(x)   )- I^*(  w(x)) \,].$$

For each fixed $w$ we will prove that $W(w,x)$ is
analytic in $x$ (in a complex neighborhood of $[0,1]$).

As  for a fixed $w$, $\,W( w ,x )$ is analytic on $x$ (see
corollary \ref{ana}), a result on piecewise analyticity of $V$  is
obtained if we are able to assume conditions to assure that $w(x)\in \Sigma$ is
locally constant as a function of $x\in[0,1]$ (up to a finite set of points $x$).
In some case there exist just a finite number of possible points $w(x)$ (see fig 2).


\vspace{0.3cm}

{\bf Section by section description of the proof}
\bigskip

In Section 2 we present some more  basic definitions and in Section 3 we show
the existence of a certain function $h_w(x)=h(w,x)$ which  defines
by means of $\log ( h(w,x))$  an involution kernel for $\log g$. In
Section 4 we present some basic results in Ergodic Optimization,
and, we   describe the main strategy for getting the piecewise
analytic sub-action $V$. Section 4 shows the relation of the scaling
function (see \cite{SS} \cite{GJQ}) with the involution kernel, and,
the potential $\log g=-\log f'$. In fact, we consider in this
section a more general setting considering any given  potential
$\log g$.  A main point we will need later is the proof of the analyticity on the variable $x$ for $w$ fixed.  This is the purpose of Section 5. In Section 6 (and also 4) we consider Gibbs states for the
potential $\beta \log g$, where $\beta$ is a real parameter. In
Section 7 and 8 we show
the existence of the piecewise complex analytic calibrated
sub-action.  The main idea is to get the piecewise analyticity for the subaction from the analyticity of the involution kernel. We need in this moment a finiteness condition for the set optimal points.
The  turning point will play an essential rule in this analysis.  In the end of this section an example shows that using our technique it is possible to  get explicit computations and to be able to exhibit a calibrated subaction in some complicated examples.

\bigskip

\begin{center}
\includegraphics[scale=0.63,angle=0]{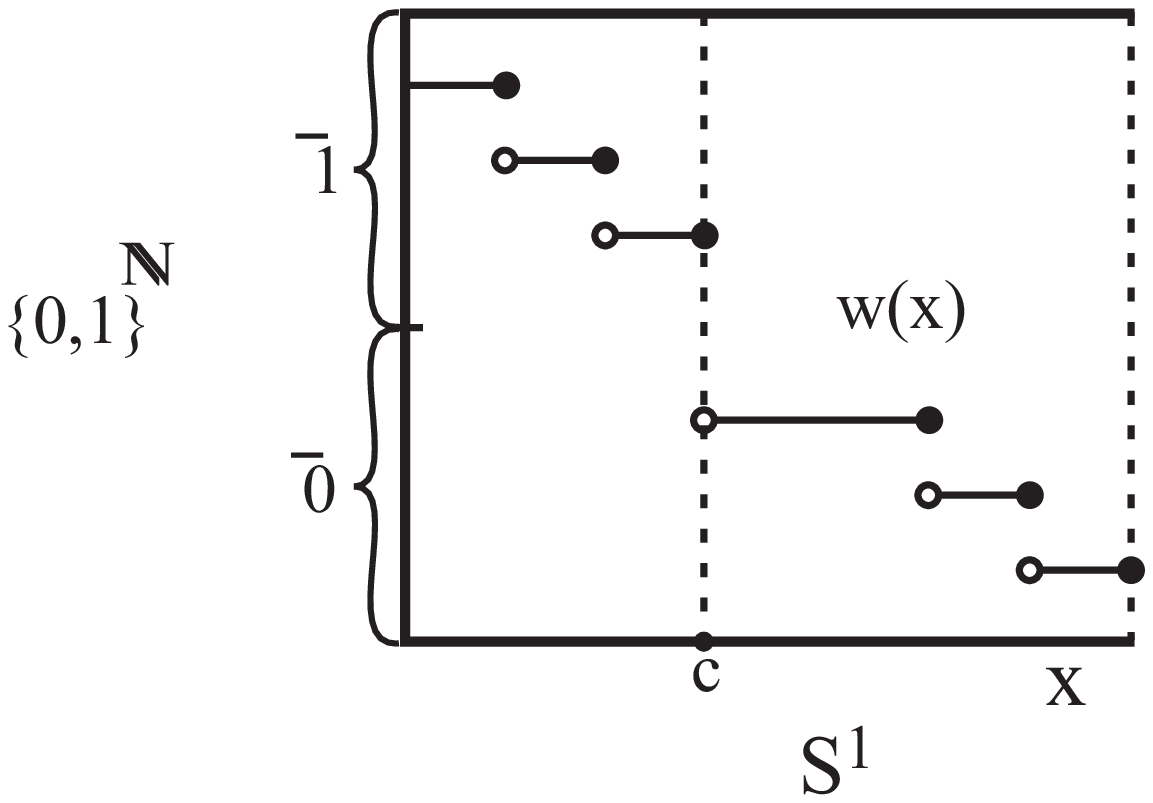}\\
\small{fig. 1 \,\, The turning point $c$}\\
\end{center}

Finally, in the
last section we present a result of independent interest for the
case where the maximizing probability is not a periodic orbit:  we
consider properties of the involution kernel for a generic $x$.

\bigskip

We will use here some ideas from Transport Theory (see \cite{Vi1} \cite{Vi2}) to show our main result. We point out that, in principle, this area has no dynamical content. But, considering a cost function (the involution kernel to be defined later) with dynamical properties one can obtain interesting properties in Ergodic Theory. The fundamental  relation (Proposition \ref{fun}) and a subsequent lemma show that the underlying dynamics spread optimal pairs for the dual Kantorovich problem. This is a special attribute of Ergodic Transport Theory. In \cite{LOT} the main issue was the understanding of points in the support of the maximizing measure. Here we focus on properties outside the support.

After this paper was written we discovered  that some of the ideas described in section 3 appeared in some form in \cite{PR} \cite{GJQ} (but, as far as we can see, not exactly like here).

\section{Onto analytic expanding maps}

 We will  consider a complex analytic extension of the real Ruelle operator $P_{\log g}$ and  general references for this topic are \cite{PP} \cite{Ru} \cite{Pol2} page 14. We describe briefly below the extension of the Ruelle operator to an action in complex functions defined in a small neighborhood of $I$.

The results we state below can be found basically  in \cite{Pol} section 2 pages 165-167 adapted to the present situation.

Denote $I=[0,1]$.  We say that $f\colon I\rightarrow I$ is an onto  map if  there exists a finite partition of $I$ by closed intervals
\begin{equation}\label{part} \{ I_i \}_{i\in \{1,2,..,d\}},\end{equation}
with pairwise disjoint interiors, such that \\
\begin{itemize}
\item[-]  For each $i$ we have that $f(I_i)=I$,\\
\item[-]  $f_i$ is monotone on each $I_i$.\\
\end{itemize}

\begin{defi}
We say that $f$ is expanding if $f$ is $C^1$ on each $I_i$ and there exists $\tilde{\lambda} > 1$ such that
$$\inf_i \inf_{x \in I_i} |Df(x)|\geq \tilde{\lambda}.$$

\end{defi}

Denote by
$$\psi_i\colon I \rightarrow I_i$$
the inverse branch of $f$ satisfying $$\psi_i\circ f(x)=x$$ for each $x \in I_i$.

We will say that an expanding onto map is analytic if there exists an simply connected, precompact  open set $O\subset \mathbb{C}$, with $I \subset O$, such that,  each $\psi_i$ has a univalent extension
$$\psi_i\colon O\rightarrow \psi_i(O).$$

{\bf We assume we can  choose} $O$ such that
\begin{itemize}
\item[-]  $\psi_i$ has a continuous extension
$$\psi_i\colon \overline{O}\rightarrow \mathbb{C}.$$
\item[-] We have
$$\psi_i(\overline{O}) \subset O.$$
\item[-] Moreover
$$\sup_i \sup_{x \in O}  |D\psi_i(x)| \leq  \lambda=\frac{\tilde{\lambda}^{-1}+1}{2}< 1.$$
\end{itemize}
Consider a finite word
$$\gamma =(i_1,i_2,\dots,i_k),$$
where $i_j \in \{1,2,..,d  \}$. Denote $|\gamma|=k$.
Define the univalent maps
$$\psi_{\gamma} \colon O \rightarrow \mathbb{C}$$
as
$$\psi_{\gamma} =  \psi_{i_k}\circ \psi_{i_{k-1}} \circ \cdots \circ \psi_{i_1},$$
We will denote
$$I_{\gamma}:= \psi_{\gamma}(I).$$
Given either an infinite word
$$\omega=(i_1,i_2,\dots,i_k,\dots) \in \Sigma:=\{1,2,..,d  \}^{\mathbb{N}},$$
or a finite word with $|\omega|\geq k$, define  its $k$-truncation as
$$\omega_k=(i_1,i_2,\dots,i_k).$$
Note that for $k\geq 1$
$$\psi_{\omega_{k}}=\psi_{i_{k}}\circ \psi_{\omega_{k-1}}.$$
For every finite word $\gamma$ we can define the cylinder
\begin{equation}\label{cyl} C_\gamma= \{ \omega \in \{1,2,..,d  \}^{\mathbb{N}}\colon \ \omega_{|\gamma|}=\gamma\}.\end{equation}

\bigskip

\begin{center}
\includegraphics[scale=0.4,angle=0]{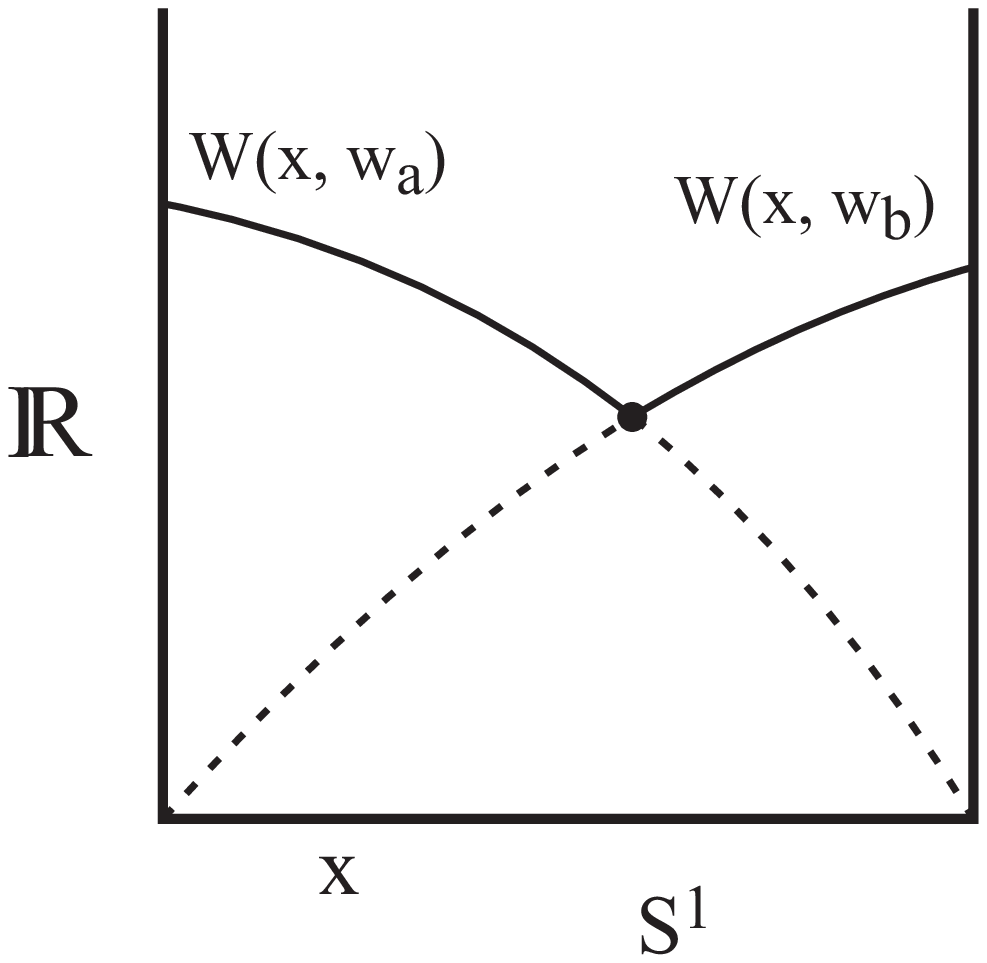}\\
\small{Fig 2) The graph of an specific example of a  piecewise analytic subaction associated to a maximizing probability which is an orbit of period $2$. It is the maximum of $W(.\,,w_a)$ and $W(.\,,w_b)$, where  $\{w_a,w_b\}\subset\{1,2\}^\mathbb{N}$ is an orbit of period $ 2$ for the shift.}\\
\end{center}

\section{Analytic potentials, spectral projections and invariant densities}
\label{anapot}

Some of the results presented in this section extend some of the ones in \cite{MM}. We say that a function
$$g\colon \cup_i \ \text{int} \ I_i \rightarrow \mathbb{R}$$
is a complex analytic potential if there are complex analytic  functions $g_i\colon \psi_i (O)\rightarrow \mathbb{C}$  such that
\begin{itemize}
\item[-] The functions $g_i$ and $g$ coincides in the interior of $I_i$.
\item[-] The functions $g_i$ have a continuous extension to $\overline{\psi_i(O)}$.
\item[-] There exists $\theta < 1$ such that
$$0< \inf_{x \in \overline{\psi_i(O)}} |g_i'(x)|\leq \sup_{x \in \overline{\psi_i(O)}} |g_i'(x)|\leq \theta.$$
\item[-] We have
$$g_i(\mathbb{R}\cap \psi_i(O)) \subset \mathbb{R}^+.$$
\end{itemize}

Denote
$$\tilde{h}_i(x)= g_i(\psi_i(x)).$$
For every finite word $\gamma$ we will define by induction on  the lengths of the words the function
$$\tilde{h}_\gamma\colon O \rightarrow \mathbb{C}$$
in the following way: Let $\gamma=(i_1,i_2,\dots,i_{k+1})$. If $|\gamma|=k+1=1$ define $\tilde{h}_{\gamma}(x)= g_{i_1}(\psi_{i_1}(x))$, otherwise
$$\tilde{h}_{\gamma}(x)= \tilde{h}_{\gamma_{k}}(x)\cdot  g_{i_{k+1}}\circ \psi_{\gamma_{k+1}}(x)=  \tilde{h}_{\gamma_{k}}(x)\cdot  \tilde{h}_{i_{k+1}}\circ \psi_{\gamma_{k}}(x).$$

As the functions we consider have complex  analytic extensions, then, $\tilde{h}_{\gamma}$ is complex analytic, but it is real when restricted to the interval $I$.

\begin{defi}
Define the Perron-Frobenious operator
$$P_{\log g}\colon C(I)\rightarrow C(I).$$
as
$$(P_{\log g} \ q)(x) =\sum_i  \tilde{h}_i(x) \  q(\psi_i(x)).$$
Note that
$$(P_{\log g}^n \ q)(x) =\sum_{|\gamma|=n}  \tilde{h}_\gamma(x) \  q(\psi_\gamma(x)).$$

\end{defi}

From \cite{PP} there exists a probability  $\tilde{\mu}$,  with no atoms and  whose support is $I$, a H\"older-continuous and positive function $v$  and $\alpha > 0$ such that
\begin{equation} \label{v} P_{\log g}^n v = \alpha^n v, \  \tilde{\mu}(v)=1,\end{equation}
and
$$\tilde{\mu} (P_{\log g}^n q)= \alpha^n \ \tilde{\mu}(q)$$
for every $q \in C(I)$. Let $v \tilde{\mu}$ be the measure absolutely continuous with respect to $\tilde{\mu}$ and whose Radon-Nikodyn derivative with respect to $\tilde{\mu}$ is $v$, that is, for every Borel set $A$ we have
$$v \tilde{\mu}(A)= \int_A v(x) \ d\tilde{\mu}(x).$$
Then the probability  $v \tilde{\mu}$ is $f$-invariant.
Let $\omega$ be either an infinite word $\omega=(i_1,i_2,\dots,i_{k},\dots)$ or a finite word with $|\omega|\geq k+n$. Then

\begin{equation}\label{relacao0}
\tilde{\mu} (I_{\omega_{k+n}})=\frac{1}{\alpha^n} \int_{I_{\omega_{k} }}  \tilde{h}_{\omega_{n+k}-\omega_k}(x)   \ d\tilde{\mu}(x),
\end{equation}

where $\omega_{n+k}-\omega_k$ is the word
$$(i_{k+1},i_{k+2},\dots,i_{k+n}).$$

The above expression is sometimes called the conformality of the probability $\tilde{\mu}$.

For every finite word $\gamma$, define
$$h_{\gamma}= \frac{\tilde{h}_{\gamma} }{\alpha^{|\gamma|} \tilde{\mu} (I_{\gamma})}.$$

 $h_{\gamma}$ is complex analytic but it is real when restricted to $I$.

Note that for $|\omega|\geq k+1$
\begin{equation}\label{rec2}  h_{\omega_{k+1}}(x)=  h_{\omega_{k}}(x)\cdot  g_{i_{k+1}}\circ \psi_{ \omega_{k+1}}(x)\frac{\tilde{\mu}(I_{\omega_{k}})}{\alpha \ \tilde{\mu}(I_{\omega_{k+1}})}=  h_{\omega_{k}}(x)\cdot  \tilde{h}_{i_{k+1}}\circ \psi_{ \omega_{k}}(x)\frac{\tilde{\mu}(I_{\omega_{k}})}{\alpha \ \tilde{\mu}(I_{\omega_{k+1}})}.\end{equation}

Let $U\subset \mathbb{C}$ be a pre-compact open set. Consider the Banach space $\mathcal{B}(U)$ of all complex analytic functions
$$h\colon U\rightarrow \mathbb{C}$$
that have  a continuous extension on $\overline{U}$, endowed with the sup norm.

The following lemma (see theorem 2.3.2 in
page 15 in \cite{MF}) is a well-known result on holomorphic functions which is very much used in complex dynamics \cite{PU} \cite{Mil}.

\begin{lem}\label{subseql} If $U,U_1 \subset \mathbb{C}$ are relatively compact open sets such that  $\overline{U}_1 \subset U$ then the inclusion $\imath \colon \mathcal{B}(U)\rightarrow  \mathcal{B}(U_1)$ is a compact linear operator. So every bounded sequence $f_n \in \mathcal{B}(U)$ has a subsequence $f_{n_i}$ such that $f_{n_i}$ converges uniformly on $\overline{U}_1$ to a continuous function that is complex analytic in $U_1$. Moreover if  $U_n$  is a sequence of open sets  such that  $\overline{U}_n \subset U$ and
$$\cup_n U_n =U,$$
we can use a diagonal  argument to show that we can find a subsequence $f_{n_i}$ and a bounded complex analytic function $f$ on $U$ such  that $f_{n_i}$  converges uniformly to $f$ on each  compact subset of $U$.
\end{lem}

\begin{thm}\label{hbeta} There exists $K > 0$ with the following property: For every  infinite word $\omega$ the sequence  $h_{\omega_k}$ is  a Cauchy sequence in $\mathcal{B}(O)$. Let $h_\omega$ be its limit. For every $\omega$ and $x \in O$ we have
$$\frac{1}{K} \leq |h_{\omega}(x)|\leq K.$$
\end{thm}
\begin{proof} Indeed since
$$\psi_{i_{k+1}}(I_{\omega_{k}})=  I_{\omega_{k+1}},$$
we have  from the conformality property (\ref{relacao0})

\begin{equation}\label{relacao} \alpha \ \tilde{\mu}(I_{\omega_{k+1}})= \int_{I_{\omega_{k}}}     g_{i_{k+1}}\circ \psi_{i_{k+1}}(y)   \ d\tilde{\mu}(y).\end{equation}
Since $g_i$ is analytic and
$$diam \ \psi_{\omega_{k+1}}(O)\leq C \lambda^{k+1},$$
by Eq. (\ref{rec2}) we have that  if $\delta_{k,x,y}$ is defined by
$$   \frac{ g_{i_{k+1}}\circ \psi_{i_{k+1}}(y) }{ g_{i_{k+1}}\circ \psi_{i_{k+1}}(x) } = 1+\delta_{k,x,y},$$
then,
$$|\delta_{k,x,y}| \leq C \lambda^{k+1}.$$
for every $x, y \in  \psi_{\omega_{k}}(O)$. Here $C$ does not depend on either $x,y \in O$, $k\geq 1$, or  $\omega$. In particular,  if $\tilde{\delta}_{k,x}$ is defined by
$$g_{i_{k+1}}\circ \psi_{ \omega_{k+1}}(x)\frac{\tilde{\mu}(I_{\omega_{k}})}{\alpha \ \tilde{\mu}(I_{\omega_{k+1}})} =1+\tilde{\delta}_{k,x},$$
then, by conformality of $\tilde{\mu}$ and the usual bounded distortion argument (for instance \cite{Man1} page 169)
$$|\tilde{\delta}_{k,x}|\leq  C \lambda^{k}.$$
for $x \in O$. This implies that for $m > n$, if $\epsilon_{n,m}$ is defined by
$$\frac{h_{\omega_m}(x)}{h_{\omega_n}(x)}=1 + \epsilon_{n,m},$$
then,
\begin{equation} \label{estmn} |\epsilon_{n,m}|\leq C_1\lambda^n\end{equation}
for some $C_1$. Here $C_1$ does not depend on  $x,y \in O$, $k\geq 1$, or  $\omega$.

Let $m_0$ large enough such that $C_1\lambda^{m_0}  < 1$. Then
$$ \inf_{y \in O, |\gamma| < m_0} |h_{\gamma}(y)| \prod_{k=m_0}^\infty (1-  C_1\lambda^k) \leq |h_{\omega_k}(x)|\leq \sup_{y \in O, |\gamma| < m_0} |h_{\gamma}(y)| \prod_{k=m_0}^\infty (1+ C_1\lambda^k)$$
for every $x \in O$, infinite word $\omega$ and $k\geq 1$.
In particular there exists $K > 0$ such that
\begin{equation}\label{estul}   \frac{1}{K} \leq |h_{\omega_k}(x)|\leq K\end{equation}
for every $k\geq 1$, $x \in O$ and infinite word $\omega$. The family $h_{\omega_k}$ is equicontinuous. Indeed,  by  estimate  (\ref{estmn})  we have that
$$\frac{h_{\omega_m}(x)- h_{\omega_n}(x)}{h_{\omega_n}(x)}= \epsilon_{n,m},$$
and by (\ref{estul}) we have that $h_{\omega_n}$ is bounded above and below. Then, we conclude that  $h_{\omega_k}$ converges.

Denote
$$h_{\omega}=\lim_k h_{\omega_k}.$$
It follows from Eq. (\ref{estul}) that \begin{equation}\label{lowerbounds}    \frac{1}{K} \leq |h_{\omega}(x)|\leq K\end{equation}       for every $x \in O$ and infinite word $\omega$.

For each $\omega$ the function $h_{\omega}$ is complex analytic. It is the extension of a strictly positive real function defined on $I$.

\end{proof}

\begin{cor} \label{anaex} For each $\omega \in \Sigma$ the function $\log h_\omega(\cdot)\colon I \rightarrow \mathbb{R}$ has a complex analytic extension to $O$.
\end{cor}
\begin{proof} Since $O$ is a simply connected open set,  the funtions  $h_\omega$ are complex analytic, and $h_\omega(x) \neq 0$ for every $x \in O$, the result follows from the property of the normal families in Complex Analysis (see \cite{Con} Cor. 6.17).
\end{proof}
We use the notation  $h_\omega(x) = h(\omega,x),$ $h_{\omega_k} (x) = h(\omega_k,x),$  for $x\in [0,1]$ and $\omega \in \{1,2,..,d\}^\mathbb{N}$, according to convenience.

For every $\tilde{\mu}$-integrable  function $z\colon I \rightarrow \mathbb{R}$ we can define the signed measure $z\tilde{\mu}$ as
$$(z\tilde{\mu})(A)=\int_A z(x) \tilde{\mu}(x)$$
for every Borel set $A \subset I$.

\begin{thm}\label{formula} Let $$z\colon I \rightarrow \mathbb{R}$$ be a positive H\"older-continuous function. Then, the sequence
$$\rho_z(x) :=  \lim_k \sum_{|\gamma|=k} h_{\gamma}(x) \,   [\,(z\,\tilde{\mu}\,)\,(I_{\gamma})]= \lim_k \sum_{|\gamma|=k} h_{\gamma}(x) \,   \int_{I_{\gamma}} z\,d\, \tilde{\mu},$$
converges for each $x \in O$. This convergence is uniform on compact subsets of $O$. Indeed
$$\rho_z(x)= v(x)\int z  \ d\tilde{\mu},$$
where $v$ is the complex analytic extension of the function $v$ defined in $(\ref{v})$.
 Furthermore, there exists a probability  $\mu$ over the Borel sigma algebra in the space of infinite words such that
 \begin{equation}\label{comconv1}
 v(x)=\rho_v(x) = \int h_{\omega}(x)\ d\mu(\omega).\end{equation}
\end{thm}

\begin{proof} Define  $\rho(k)\colon O \rightarrow \mathbb{C}$ as $$\rho(k)(x):= \,\sum_{|\gamma|=k} h_{\gamma}(x) \,\,  \int_{I_{\gamma}} z\,d\, \tilde{\mu}.$$ Firstly  we  will prove  that
\begin{equation}\label{comconv}  \rho(k)(x)\rightarrow_k v(x)\int z  \ d\tilde{\mu},\end{equation}  for each   $x \in I$. Indeed for $x\in I$

$$\sum_{|\gamma|=k} h_{\gamma}(x) \, \int_{I_{\gamma}} z\,d\, \tilde{\mu} = \sum_{|\gamma|=k} h_{\gamma}(x) \  z(\psi_{\gamma}(x))(1+\epsilon_{x,\gamma}) \tilde{\mu}(I_{\gamma})$$
$$=\sum_{|\gamma|=k} h_{\gamma}(x) \  z(\psi_{\gamma}(x)) \tilde{\mu}(I_{\gamma}) + \tilde{\epsilon}_{x,k}$$
$$= \alpha^{-k} \sum_{|\gamma|=k} \tilde{h}_{\gamma}(x)\  z(\psi_{\gamma}(x)) +   \tilde{\epsilon}_{x,k}$$
$$= \alpha^{-k} (P_{\log g}^kz)(x) + \tilde{\epsilon}_{x,k} .$$
Here,
$$|\epsilon_{x,\gamma}|, |\tilde{\epsilon}_{x,k}| \leq C\eta^k,$$
for some $\eta < 1$.

It is a well know fact  that
$$\lim_k  \alpha^{-k} (P_{\log g}^kz)(x)= v(x)  \int z \ d\tilde{\mu}.$$
So
$$\lim_k \rho(k)(x)= v(x)  \int z \ d\tilde{\mu}.$$
for $x \in I$.

Next we claim that  $\rho(k)$ converges uniformly on compact subsets  of $O$ to a complex analytic function $\rho_z$. Note  that  by Eq. (\ref{estul}) we have
$$|\sum_{|\gamma|=k} h_{\gamma}(x) \,   \int_{I_{\gamma}} z\,d\, \tilde{\mu}|
\leq  K\, \sup_{x \in I} |z(x)| \sum_{|\gamma|=k}  \tilde{\mu}(I_{\gamma})  \leq  K\, \sup_{x \in I} |z(x)|,$$
for every $x \in O$,
so in particular the complex analytic functions
$\rho(k)$
are uniformly bounded in $O$. By Lemma \ref{subseql} every subsequence of $\rho(k)$ has a subsequence  that converges uniformly on compact subsets of $O$  to a complex analytic function defined in $O$, so to prove the claim it is enough to show that every subsequence of $\rho(k)$ that converges uniformly on compact subsets of $O$ converges to  the very same complex analytic function. Indeed we already proved that such limit functions must coincide with $$v(x)  \int z \ d\tilde{\mu}$$ on $I$. Since the limit functions are complex analytic, if they coincide on $I$ they must coincide everywhere in $O$. This finishes the proof of the claim.  In particular taking $z(x)=1$ everywhere, this proves that $v\colon I \rightarrow \mathbb{R}$ has a complex analytic extension $v\colon O\rightarrow \mathbb{C}$. Consequently for every function $z$
$$\rho_z(x)= v(x)  \int z \ d\tilde{\mu},$$
once  we already know that these functions coincide on $I$. For any given $z$ we have that  $\rho_z(x)= v(x)\int z  \ d\tilde{\mu}$ is an eigenfunction of the Ruelle operator. So we got a spectral projection in the space of eigenfunctions.

Now  we will prove the second statement. Consider the unique probability $\mu$ defined on the space of infinite words such that on the cylinders $C_{\gamma}$, $|\gamma|< \infty$, it   satisfies
$$\mu(C_{\gamma})=(v\tilde{\mu})(I_{\gamma})=  \int_{I_{\gamma}} \,v\,d\, \tilde{\mu}.$$
Note that $\mu$ extends to a measure  on the space of infinite words  because $v\tilde{\mu}$ is $f$-invariant and it has no atoms. For  each fixed $x \in O$, the functions $\omega\rightarrow h_{\omega_k}(x)$ are constant on each cylinder $C_\gamma$, $|\gamma|=k$. So
$$\int h_{\omega_k}(x) \ d\mu(\omega)= \sum_{|\gamma|=k} h_{\gamma}(x)\mu(C_\gamma).$$
By the Dominated Convergence Theorem
$$ \int h_{\omega}(x) \ d\mu(\omega)= \lim_k \int h_{\omega_k}(x) \ d\mu(\omega) = \lim_k \sum_{|\gamma|=k} h_{\gamma}(x) \  \mu(C_\gamma)=$$
$$\lim_k \sum_{|\gamma|=k} h_{\gamma}(x) \  (v\tilde{\mu})(I_{\gamma}).$$
\end{proof}

\begin{cor} \label{cor} The function $\rho_z= v(x)\int z  \ d\tilde{\mu}$ is a $\alpha$-eigenfunction of $P_{\log g}$
$$P_{\log g}( \rho_z )= \alpha \cdot \rho_z.$$
\end{cor}

Therefore, any $\rho_z$ is an eigenfunction for the Ruelle operator for $A=\log g$. Later we will consider a real parameter $\beta$ and we will denote by $\phi_\beta(x)$ a specific normalized eigenfunction of the Ruelle operator
for $\beta \log g$.

The two results described above are in some sense similar to the ones in \cite{PU} section 9, \cite{MM}, \cite{BLT}.  We explain this claim in a more precise way in  the next section.

The results described in this section correspond in \cite{MM} to the potential $\log g = A = -\log f'$.

\section{Maximizing probabilities, the dual potential and Scaling functions }\label{dualpot}

\bigskip

From Corollary \ref{cor}, given $A=\beta\,\log g$, there exists  $\alpha_\beta$ and $  \rho_\beta  $, such that, $P_{\beta\,\log g} (\rho_\beta ) = \alpha_\beta\, \rho_\beta,$  where $\rho_\beta$ has a complex analytic extension to a neighborhood $O_\beta$.
The $\phi_{\beta A} $ is colinear with $\rho_\beta$ and satisfies the normalization described above. Therefore, we get from Corollary \ref{cor} the expression
$$\rho_\beta(x) = \int h_{\omega}(x)\ d\mu(\omega).$$

Our main purpose in this section is to get the following:

\begin{pro} \label{main1} For any $\beta$ we have that
$\log h_\omega (x) =\log  h_\beta(\omega,x)$ is well defined and is an involution kernel for $\beta \log g$. For $\omega$ and $\beta$ fixed, the function $\log h_\beta(\omega,.)$ has a  complex analytic extension to  a complex  neighborhood $O$ of $[0,1]$.
\end{pro}

Given a finite word $\gamma=(i_1,i_2,\dots,i_k)$, $k> 1$, define $\sigma^\star(\gamma)=(i_2,\dots,i_k)$. For infinite words we define $\sigma^\star$ as the usual shift function.   The {\em scaling function } $s\colon \Sigma \rightarrow \mathbb{R}$  of the potential $g$ is defined as
$$s(\omega)= \lim_{k\rightarrow \infty} \frac{\tilde{\mu}(I_{\omega_k})}{\tilde{\mu}(I_{\sigma^\star(\omega_k)})}.$$

This definition is the natural generalization of the scaling function  in \cite{SS} and \cite{GJQ}. If we take  $\log g=-\log f'$ then we get their result. It will follow from our results  the existence of an involution kernel which provides a co-homology between the scaling function $[\,\log  (\alpha \, s)\,](w)$ and $\log g(x)=-\log f'(x)$.  The constant $\alpha$ is  the eigenvalue defined before in section 1.

To verify that the above limit indeed exists, note that by Eq. (\ref{relacao})  and since $g$
is a H\"older-continuous function we have that

$$\frac{\tilde{\mu}(I_{\omega_{k+1}})}{ \tilde{\mu}(I_{\sigma^\star(\omega_{k+1})})}= \frac{\int_{I_{\omega_{k}}}
g\circ \psi_{i_{k+1}}(y)   \ d\tilde{\mu}(y)}{ \int_{I_{\sigma^\star(\omega_{k})}}     g\circ \psi_{i_{k+1}}(y)   \ d\tilde{\mu}(y)}=(1+\epsilon_k) \frac{\tilde{\mu}(I_{\omega_{k})}}{ \tilde{\mu}(I_{\sigma^\star(\omega_{k})})},$$
where $|\epsilon_k|\leq C\lambda^k$. So $s(\omega)$ is well defined.

Note that, since $v >0$ is a H\"older function and  $I_{\omega_k} \subset I_{\sigma(\omega_k)}$,
$$s(\omega)=  \lim_{k\rightarrow \infty} \frac{(v\tilde{\mu})(I_{\omega_k})}{(v\tilde{\mu})(I_{\sigma^\star(\omega_k)})} = \lim_{k\rightarrow \infty} \frac{\mu(C_{\omega_k})}{\mu(C_{\sigma^\star(\omega_k)})},$$
so the the scaling function $s$ is the Jacobian of the measure $\mu$.

The {\em dual potential} $g^\star$ is defined as
$$g^{\star}(\omega):=\alpha s(\omega).$$

\begin{lem} We have that
$$\frac{g^{\star}(\omega)}{g(\psi_{i_0}(x))}= \frac{h(\sigma(\omega), \psi_{i_0}(x))}{h(\omega,x)}.$$
\end{lem}

\begin{proof} Indeed
$$\frac{h(\sigma(\omega), \psi_{i_0}(x))}{h(\omega,x)}$$
$$= \lim_k \frac{h(\sigma(\omega_k), \psi_{i_0}(x))}{h(\omega_k,x)}.$$
$$\lim_k \frac{\tilde{h}(\sigma(\omega_k), \psi_{i_0}(x))}{\tilde{h}(\omega_k,x)}\frac{\alpha^{k}\tilde{\mu}(I_{\omega_k})}{\alpha^{k-1}\tilde{\mu}(I_{\sigma^*(\omega_k)})} $$
$$=\lim_k \frac{\alpha}{ g(\psi_{i_0}(x))}  \frac{\tilde{\mu}(I_{\omega_k})}{\tilde{\mu}(I_{\sigma(\omega_k)})}   $$
$$=\frac{\alpha}{ g(\psi_{i_0}(x))}  s(\omega) $$

\end{proof}

From the above we finally get Proposition \ref{main1}.

\section{Analyticity of the involution kernel}

From last section we  get that for each value $\beta\geq 0$
$$ \rho_{\beta A}(x) = \int e^{ W_\beta(w,x)} \,d
\nu_{\beta A^*}(w)= \int \, h_\beta (w,x)\,  \,d
\nu_{\beta A^*}(w),$$
is an eigenfunton for the Ruelle operator of the potential $\beta \log g$. The involution kernel  $W_\beta$ depends of the variable $\beta$.
\bigskip

{\bf Remark:}
There is a main difference from the reasoning of this section to the procedures in     \cite{BLT}. We will explain this.
Suppose $W_1$ is an involution kernel for $\log g$ (that is, $\beta=1$). Therefore, given a real value $\beta$ we have
$$\beta\, (\log g)^*(w)= \beta \,\log \, g \circ \T^{-1}(w,x)+ \beta \,W_1 \circ \T^{-1}(w,x) -
\beta \, W_1(w,x).$$

The involution kernel is not unique (see \cite{BLT}). We point out that $\log (h_\beta (w,x))$ is not necessarily equal to $\beta W_1$. This will require an extra work.
We will need to show the existence of a $H_\infty(w,x)$ (complex analytic on $x$), such that,
$h_\beta(w,x) \sim e^{ \beta \, H_\infty(w,x)}$ (in the sense that $\lim_{\beta \to \infty} \frac{1}{\beta} \log h_\beta(w,x)  = H_\infty(w,x)$).  In other words, we want to replace $W_\beta$ by a $\beta H_\infty$ (in the notation that will be followed later).
\bigskip

 We will show in Corollary \ref{ana}  that for each fixed $w$ the family $\frac{1}{\beta} \log h_\beta(w,x)$, $\beta>0$,
 is  normal.
\bigskip

Remember that for a given $w\in \Sigma$, we have
$h_{\omega}=\lim_k h_{\omega_k}.$

\begin{pro}\label{converge} Let $K \subset O$ be a compact.  There exists $C$  such that the following holds:
\begin{itemize}
\item[A.] For every $\beta \geq 1$ and $x\in K$, $\omega        \in \Sigma$, we have
\begin{equation}\label{est0}e^{-\beta C}\leq | h_\beta(\omega_1,x) |\leq e^{\beta C}\end{equation}
\item[B.] For every $\beta \geq 1$, $x \in K$, $\omega \in \Sigma$ and $k\geq 1$  we have
\begin{equation}\label{estabs} e^{-C\beta\lambda^{k}}\leq \Big|\frac{ h_\beta(\omega_{k+1},x)}{ h_\beta(\omega_{k},x)}\Big| \leq  e^{C\beta\lambda^{k}}.\end{equation}
\item[C.] For every finite word $\gamma$ there is a  function
$$q_{\gamma}\colon \mathbb{R}\times O\rightarrow \mathbb{C},$$
that is holomorphic on $x$, real valued for $x \in \mathbb{R}$ and  which does not depend on $K$, such that for every $x \in O$, $\beta \geq 1$, $\omega \in \Sigma$  we have
\begin{equation} \label{rec} h_\beta(\omega_{k+1},x)=e^{q_{\omega_{k+1}}(\beta,x)}.\end{equation}
Furthermore
\begin{equation}\label{estargini} |q_{\omega_{1}}(\beta,x)|\leq C\beta \end{equation}
and
\begin{equation}\label{estarg} | q_{\omega_{k+1}}(\beta,x)-  q_{\omega_{k}}(\beta,x) |\leq  C\beta\lambda^{k}\end{equation}
for every $\beta \geq 1$, $x \in K$, $\omega \in \Sigma$ and $k\geq 1$.
\end{itemize}
\end{pro}
\begin{proof}[Proof of Claim A] Recall that for $i \in \{1,\dots,d\}$
\begin{equation}\label{primiero} h_\beta(i,x)= \frac{g_i^\beta (\psi_i(x))}{\alpha \tilde{\mu}_\beta(I_i)}=\frac{g_i^\beta (\psi_i(x))}{\int_I g_i^\beta (\psi_i(y)) \tilde{\mu}_\beta(y)}, \end{equation}
so
$$|h_\beta(i,x)|= \frac{1}{\int_I \frac{g_i^\beta (\psi_i(y))}{|g_i^\beta (\psi_i(x))|} \tilde{\mu}_\beta(y)}.$$
Since $g_i$ are  holomorphic on $\psi_i(O)$, $g_i\not=  0$ in $\psi_i(O)$, for every compact $K \subset O$ there exists $C_1$ such that
\begin{equation}\label{esta}  e^{-C_1} \leq \frac{|g_i(\psi_i(x))|}{|g_i(\psi_i(y))|}\leq e^{C_1}\end{equation}
for every $x,y \in K$ and $i$. Since $\tilde{\mu}_\beta(I)=1$, it is now easy to obtain Eq. (\ref{est0}).
\end{proof}

\begin{proof}[Proof of Claim B] Since $g_i$ are  holomorphic on $\psi_i(O)$, $g_i\not=  0$ in $\psi_i(O)$, for every compact $K \subset O$ there exists $C_2$ such that
\begin{equation}\label{est}  e^{-C_2|x-y|} \leq \Big|\frac{g_i(\psi_i(x))}{g_i(\psi_i(y))}\Big|\leq e^{C_2|x-y|}\end{equation}
for every $x,y \in K$ and $i$.  Note that every such compact is contained in a larger compact set $\tilde{K} \subset O$ such that $\psi_i(\tilde{K}) \subset \tilde{K}$ for every $i$, so we can assume that $K$ has this property. Let $x \in K$. By  Eq. (\ref{relacao})
\begin{align} \nonumber  \frac{h_\beta(\omega_{k+1},x)}{h_\beta(\omega_{k},x)}
=&\frac{\tilde{h}^\beta(\omega_{k+1},x)}{\tilde{h}^\beta(\omega_{k},x)}\frac{\alpha^k_\beta\tilde{\mu}_\beta ( I_{\omega_k})}{\alpha^{k+1}_\beta\tilde{\mu}_\beta( I_{\omega_{k+1}})}\\
\nonumber =& \frac{g^\beta_{i_{k+1}}(\psi_{\omega_{k+1}}(x))}{\alpha_\beta}\frac{\tilde{\mu}_\beta( I_{\omega_k})}{\tilde{\mu}_\beta( I_{\omega_{k+1}})}\\
\nonumber =& \frac{g^\beta_{i_{k+1}}(\psi_{\omega_{k+1}}(x))}{\alpha_\beta}\frac{\alpha_\beta \tilde{\mu}_\beta(I_{\omega_k})}{   \int_{I_{\omega_{k}}}     g_{i_{k+1}}^\beta\circ \psi_{i_{k+1}}(y)   \ d\tilde{\mu}_\beta(y)   }\\
\label{formula}  =&\frac{g^\beta_{i_{k+1}}(\psi_{\omega_{k+1}}(x)) \tilde{\mu}_\beta(I_{\omega_k})}{   \int_{I_{\omega_{k}}}     g_{i_{k+1}}^\beta\circ \psi_{i_{k+1}}(y)   \ d\tilde{\mu}_\beta(y)   }\\
\label{estabsf} =& \frac{ \tilde{\mu}_\beta(I_{\omega_k})}{   \int_{I_{\omega_{k}}}     \frac{g_{i_{k+1}}^\beta\circ \psi_{i_{k+1}}(y)}{g^\beta_{i_{k+1}}(\psi_{\omega_{k+1}}(x))}   \ d\tilde{\mu}_\beta(y)   }.\end{align}
In particular
$$\Big|\frac{h_\beta(\omega_{k+1},x)}{h_\beta(\omega_{k},x)}\Big|=
 \frac{ \tilde{\mu}_\beta(I_{\omega_k})}{   \int_{I_{\omega_{k}}}     \frac{g_{i_{k+1}}^\beta\circ \psi_{i_{k+1}}(y)}{|g^\beta_{i_{k+1}}(\psi_{\omega_{k+1}}(x))|}   \ d\tilde{\mu}_\beta(y)   }.$$
For every $y \in I_{\omega_k}$ we have  $$\psi_{i_{k+1}}(y),\ \psi_{\omega_{k+1}}(x) \in \psi_{\omega_{k+1}}(O)$$
From Eq. (\ref{est}) we obtain
$$e^{-C\beta \lambda^{k}}\leq e^{-C\beta \ diam \ \psi_{\omega_{k+1}}(O) }   \leq \frac{g_{i_{k+1}}^\beta\circ \psi_{i_{k+1}}(y)}{|g_{i_{k+1}}^\beta (\psi_{\omega_{k+1}}(x))|}\leq e^{C\beta \ diam \ \psi_{\omega_{k+1}}(O)}\leq e^{C\beta \lambda^{k}}$$
So
$$e^{- C\beta \lambda^{k}}\leq \frac{|h_\beta(\omega_{k+1},x)|}{|h_\beta(\omega_{k},x)|}\leq e^{C\beta \lambda^{k}}.$$
\end{proof}

\begin{proof}[Proof of Claim C]

Since $g_i\circ \psi_i \colon O \rightarrow \mathbb{C}$ does not vanish and $O$ is a  simply connected domain, there exists a (unique) function  $r_i\colon O \rightarrow \mathbb{C}$ such that $g_i\circ \psi_i =e^{r_i}$ on $O$ and $Im \ r_i(x)=0$ for $x \in \mathbb{R}$. Since $\psi_{\gamma}(O)\cap I \not= \emptyset$ and $diam \ \psi_{\gamma}(O)\leq \lambda^{|\gamma|}$ we have that
\begin{equation}\label{ima} |Im \ r_i(\psi_{\gamma}(x))|\leq C_3 \lambda^{|\gamma|}\end{equation}
for every $x \in O$ and every finite word $\gamma$.

Define $$q_i(\beta,x)=\beta r_i(x)+ \log \frac{1}{   \int_{I_i}     g_{i}^\beta\circ \psi_{i}(y)   \ d\tilde{\mu}_\beta(y)   }.$$ and $q_{\gamma}$, with $\gamma=(i_1,\dots,i_{k+1})$,  by induction  on $k$,  as

$$q_{\gamma}(\beta,x)= q_{\gamma_k}(\beta, x) + \beta r_{i_{k+1}}( \psi_{\gamma_k}(x))+ \log \frac{\tilde{\mu}_\beta(I_{\gamma_k})}{   \int_{I_{\gamma_{k}}}     g_{i_{k+1}}^\beta\circ \psi_{i_{k+1}}(y)   \ d\tilde{\mu}_\beta(y)   }.$$
It follows from Eq. (\ref{formula}) that $q_{\gamma}$ satisfies Eq. (\ref{rec}), so
$$Re \ q_\gamma(\beta,x)=\log |h_\beta(\gamma,x)|,$$
in particular by Eq. (\ref{est0}) e (\ref{estabs}) we have
\begin{equation}\label{estreini} |Re \ q_{\omega_{1}}(\beta,x)|\leq C_4\beta \end{equation}
and
\begin{equation}\label{estre} |Re \ q_{\omega_{k+1}}(\beta,x)- Re \ q_{\omega_{k}}(\beta,x) |\leq  C_5\beta\lambda^{k}\end{equation}

for $\beta\geq1$. Furthermore for every $\beta \in \mathbb{R}$, $\omega \in \Sigma$ and $k\geq 1$
$$|Im \ q_{\omega_{k+1}}(\beta,x) - Im \ q_{\omega_k}(\beta, x)|=     |\beta \ Im \ r_{i_{k+1}}( \psi_{\omega_k}(x))|\leq C_6|\beta| \lambda^{k}.$$
Moreover for $\beta > 0$ we have
$$|Im \ q_i(\beta,x)|= |\beta| |Im \ r_i(\psi_{\omega_k}(x))|\leq C_7 |\beta|.$$
\end{proof}

For  every $x \in O$ define
$$H_{\beta,k}(\omega,x):=  \frac{1}{\beta}q_{\omega_k} (\beta,x).$$
In particular, if $x \in I$ we have that $h_\beta(\omega_k,x)$ is a nonnegative real number by our choice of the branches $r_i$, so
$$H_{\beta,k}(\omega,x)=  \frac{1}{\beta}\log h_\beta(\omega_k,x)$$
for $x \in I$. It follows from Proposition \ref{converge} that for every compact $K \subset O$ there exists  $D$ such that
\begin{equation}\label{estdiv1} |H_{\beta,1}(\omega,x)|\leq D,\end{equation}
\begin{equation}\label{estdiv2}  |H_{\beta,k+1}(\omega,x)-H_{\beta,k}(\omega,x)|\leq D\lambda^{k} \end{equation}
for $x \in K$,   and every $k$ and $\omega$. So there exists some constant $C_8$ such that
$$|H_{\beta,k}(\omega,x)|\leq C_8$$
for every $k$, $\omega$, $x \in K$.  This implies that the family of functions
$$\mathcal{F}_1= \{ H_{\beta,k}(\omega,\cdot)\}_{k,\omega, \beta \geq 1}$$
is a normal family on $O$, that is, every sequence of functions in this family admits a subsequence that converges uniformly on every compact subset of $O$. In Theorem \ref{hbeta} we showed that for every $x \in I$ we have
$$\lim_k h_\beta(\omega_k,x)=   h_\beta(\omega,x)> 0,$$
so
$$\lim_k H_{\beta,k}(\omega,x)=   \frac{1}{\beta} \log h_\beta(\omega,x),$$
for $x \in I$. It follows from the normality of the family $\mathcal{F}$ that the limit
$$H_{\beta}(\omega,x) :=\lim_k H_{\beta,k}(\omega,x)$$
exists for every $x \in O$ and that this limit is uniform on every compact subset of $O$. Moreover
$$\mathcal{F}_2= \{ H_{\beta}(\omega,\cdot)\}_{\omega, \beta\geq 1}$$
is also a normal family on $O$.

 We consider in $\Sigma$ the metric $d$, such that $d(\omega,\gamma)=2^{-n}
$, where $n$ is the position of the  first symbol in which $\omega$ and $\gamma$ disagree.

\begin{cor}\label{equi} For every compact $K \subset O$ there exists $C_9$ such that
\begin{equation}\label{lipbeta} |H_\beta(\omega,x)- H_\beta(\gamma,y)|\leq C_9|x-y|+ C_9d(\omega,\gamma)\end{equation}
for every $x,y \in K$.
\end{cor}
\begin{proof} Since the family $\mathcal{F}_2$ is uniformly bounded on each compact set $K \subset O$, we have that the family of functions
$$\mathcal{F}_3 := \{  H_{\beta}'(\omega,\cdot)\}_{\omega, \beta\geq 1}$$
has the same property, so it is easy to see that for every compact $K \subset O$ there exists $C$ such that
$$|H_\beta(\omega,x)- H_\beta(\omega,y)|\leq C_{10}|x-y|.$$
Note also that Eq. (\ref{estdiv2}) implies
$$|H_\beta(\omega,x)-H_\beta(\omega_k,x)|\leq C_{11}\lambda^{k},$$
Let $ k+1=  \log(d(\gamma,\omega))/\log \lambda$. Then $\gamma_k=\omega_k$ and  we have
$$|H_\beta(\omega,y)- H_\beta(\gamma,y)|\leq |H_\beta(\omega,y)- H_\beta(\omega_k,y)| + |H_\beta(\gamma_k,y)-H_\beta(\gamma,y)|\leq C_{12}d(\omega,\gamma).$$
\end{proof}

\begin{cor} \label{subseq} There exists a sequence $\beta_n > 0$ satisfying  $\beta_n \rightarrow \infty$ when $n \to \infty$ such that  the limit
\begin{equation}\label{limite} H_\infty(\omega,x)= \lim_{n\rightarrow \infty} H_{\beta_n}(\omega,x),\end{equation}
exists for every $(\omega,x)$ in
$$ \{1,\dots,d\}^\mathbb{N}\times O.$$
Moreover for every compact $K \subset O$ there exist $C_{13}$ such that
\begin{equation}\label{lip} |H_\infty(\omega,x)- H_\infty(\gamma,y)|\leq C_{13}|x-y|+ Cd(\omega,\gamma)\end{equation}
and the limit in Eq. (\ref{limite}) is uniform with respect to $(\omega,x)$ on
 \begin{equation}\label{set} \{1,\dots,d\}^\mathbb{N}\times K\end{equation}
In particular for each $\omega$ we have that  $x\rightarrow H_\infty(\omega,x)$ is holomorphic on $O$.
\end{cor}
\begin{proof}By Corollary \ref{equi},  the family of functions $H_\beta$ is equicontinuous on each set of the form (\ref{set}), where $K$ is a compact subset of $O$. So given a compact $K \subset O$ and  any sequence $\beta_j\rightarrow +\infty$, as $j \to \infty$,  there is a subsequence $\beta_{j_i}$ such that the limit
$$\lim_{i\rightarrow \infty} H_{\beta_{j_i}}(\omega,x)$$
exists and it is uniform on the set of the form (\ref{set}). Then, choosing an exhaustion by compact sets of $O$ and using Cantor's diagonal argument  we can find a sequence $\beta_n \rightarrow +\infty$ such that the limit
$$H_\infty(\omega,x)=\lim_{n\rightarrow +\infty} H_{\beta_{n}}(\omega,x)$$
exists and it is uniform on every  set of the form (\ref{set}), with compact $K \subset O$.
 Eq. (\ref{lip}) follows directly from Eq. (\ref{lipbeta}).
\end{proof}

This shows the main result in this section:

\begin{cor} \label{ana} For any $w$ fixed,  $H_\infty(\omega,x)$ is analytic on $x$.
\end{cor}

From Corollary 5.2 (the convergence is uniform) and from (\ref{comconv1})
$$\rho_v(x) = \int h_{\omega}(x)\ d\mu(\omega),$$
we get that for any $x\in[0,1]$
$$V(x) = \lim_{\beta \to +\infty} \frac{1}{\beta_n} \log \phi_{\beta_n} (x) =
\sup_{ w \in \Sigma}\, (H_\infty (w,x)  -\, I^*(w)).
$$

\bigskip

\begin{pro} The function  $H_\infty (w,x) $ is an involution kernel for $g$.
\end{pro}

\begin{proof}
 Consider $g$ fixed. Let $\beta_n$ be
 a sequence as in Corollary \ref{subseq}. For any $\beta_n$ we have
$$\frac{(g^{\beta_n})^{\star}(\omega)}{g^{\beta_n}(\psi_{i_0}(x))}= \frac{h_{\beta_n}(\sigma(\omega), \psi_{i_0}(x))}{h_{\beta_n}(\omega,x)}.$$

Taking $\frac{1}{\beta_n} \log$ in both sides and taking the limit $n\rightarrow +\infty$  we get that
$$   g (\T^{-1}(\omega,x) ) \,+ \, H_\infty  (\T^{-1}(\omega,x) )  \,     -\,      H_\infty (\omega,x) $$
depends only in the variable $w$.

Therefore, $H_\infty (w,x) $ is an involution kernel (see Remark \ref{re}).
\end{proof}

\section{A piecewise analytic subaction} \label{secmain}

We suppose in this section that
the maximizing probability for $A=\log g$ is unique (then the same happen for $A^*$, see \cite{CLT}) in order
we can define the deviation function $I^*$.

Given the analytic involution kernel $H_\infty (w,x) $ and a fixed calibrated $V^*$ (unique up to additive constant)   define  $W(w,x) =H_\infty (w,x)+ V^*(w) $. We point out that  $W$ is also analytic on the variable $x\in(0,1)$ for each $w$ fixed).
\bigskip

The reason for the introduction of such $W$ (and not $H_\infty$) is that, in this section,  instead of

$$\gamma+ V(x)  =
\sup_{ w \in \Sigma}\, [H_\infty (w,x)  -\, I^*(w)],
$$
it will be more convenient the expression
$$\gamma+ V(x) =
\sup_{ w \in \Sigma}\, [\,(W(w,x)- I^*(w))\, - \,V^*
(w)\,].
$$

We assume without lost of generality that the above $\gamma$ (see \cite{BLT} \cite{LOT})  is zero.

For each $x$ we get one (or, more) $w(x)$ such attains the supremum above by compactness. Therefore,
$$V(x) =
W(w(x),x) - V^*
(w(x))- I^*(w(x))\,.$$

\vspace{0.5cm}

If there exists
$\tilde{ w}$ such that for all $x\in(a,b)$
$$V(x) =
\sup_{ w \in \Sigma}\, (H_\infty (w,x)  -\, I^*(w))= H_\infty
(\tilde{ w},x)  -\, I^*(\tilde{ w})= W(\tilde{ w},x) - V^*
(\tilde{ w})- I^*(\tilde{ w}),
$$
then $V$ is analytic on $(a,b)$.

Let us consider for a moment the general case ($A$ not necessarily twist) .

We denote by $M$ the support of $\mu_{\infty\, A}^*$.

 As $I^*$ is lower semicontinuous and
$W- V^*$ is continuous, then for each fixed $x$, the supremum of
$H_\infty (w,x)  -\, I^*(w)$ in the variable $w$ is achieved, and we
denote (one of such $w$) it by $w(x)$. In this case we say  $w(x)$ is
optimal for $x$. We also say that $(w(x),x)$ is   an optimal pair
of points $x\in[0,1], w(x)\in \{0,1\}^\mathbb{N}$. One can ask if
this $w(x)$ is independent of $x$, and equal to a fixed $\tilde{ w}$.
This would imply that $V$ is analytic. If for all $x$ in  a certain
open interval $(a,b)$, the $w(x)$ is the same, then $V$ is analytic
in this interval. We will show under some restrictions that given
any $x$ we can find a neighborhood $(a,b)$ of $x$ where this is the
case. The number of possible intervals can be infinite. We will give
later a characterization when it is finite or infinite.

Note that given $x$, any optimal $w(x)$  satisfies $I^*(w(x))$ is
finite (otherwise a $w$ with finite $I^*(w)$ will be better). This is a strong restriction in the set of possible $w(x)$, because if $I^*(w)$ is finite, then the $\omega$-limit of $w$ have to be in the support of $\mu_{\infty A^*}$ (see section 5 \cite{LMST}).

\bigskip

\begin{example}\label{ex0}
We present examples of optimal pairs.

If $\hat{\mu}_{max}$ is the
natural extension of the maximizing probability $\mu_{\infty\, A}$, then
for all $(p^*,p)$ in the support of $\hat{\mu}_{max}$ we have the following expression taken from Proposition 5 in \cite{BLT}

$$ V(p)\, + \,V^* (p^*)\,= \,  W(p^*,p )  \,.$$

If  $(p^*,p)$ in the support of $\hat{\mu}_{max}$ (then,
$p\in [0,1]$  is in the support of $\mu_{\infty\, A}$  and  $p^*\in \Sigma$  is in the support of $\mu_{\infty\, A}^*$), then
$$V(p) =  \sup_{ w \in \Sigma}\, \, W(w,p) - V^*
( w)- I^*(w) \,=$$
$$ \, W( p^*,p)  - V^*
( p^*)- I^*(p^*)\,=          \, W( p^*,p) - V^*
( p^*)\,.
$$

Therefore, $(p^*,p)$ is an optimal pair if $(p^*,p)$ is in the support of $\hat{\mu}_{max}$. That is, $w(p)=p^*.$

If the potential $\log g$ is twist, then for any given $p$ in the support of $\mu_{\infty\, A}$, there is only one $p^*$, such that $(p,p^*)$ is in the support of $\hat{\mu}_{max}$ (see \cite{LOT}) up to one orbit. If the maximizing probability for $A$ is a periodic orbit, then the $p^*$ associated to a $p$ is unique.

\end{example}

\bigskip

In order to simplify the notation we assume that
$m(A^*)=0$.

If we denote
\begin{equation}\label{est1}
R^*(w) =  V^*\circ\sigma(w) -V^*(w)-A^* (w),
\end{equation}
then we know that $R^*\geq 0$ because $V^*$ is calibrated.

Note that the main  result in \cite{BLT} claims that the explicit expression of the deviation function is
\begin{equation}\label{est}
I^*(w)=\sum_{n\geq0} R^*\, (\,
\sigma^n(w)\,).
\end{equation}

\vspace{0.2cm}

Given $A$, we denote for $x,x'\in [0,1]$ and $w\in \Sigma$
$$\Delta(x,x',w)=\sum_{n\geq1}A\circ\psi_{w,n}(x)
-A\circ\psi_{w,n}(x').$$

The involution kernel $W$ can be computed for any $(w,x)$ by
$W(w,x)=\Delta_A(x,x',w)$,  where we choose a point $x'$ for good [CLT].

Note that for any $x,x',w$, we
have that $W(w,x)-W(w,x'')=\Delta(w,x,x'')$.

\vspace{0.2cm}

Given $A$,
suppose $R$ satisfies
$$R(x) =  V\circ f(x) -V(x)-A (x),$$ where $V$ is a calibrated subaction.
Consider a fixed involution kernel $W$. The next result
(which does not assume the twist condition) claims that the dual of
$R$ is $R^*$, and the corresponding involution kernel is $(V^* + V -
W).$

\begin{pro} \label{fun} (Fundamental Relation)(FR)

$$ R(\psi_{w}(x))= (V^* + V - W)(w,x) - (V^* + V - W)(  \sigma(w), \psi_{w}(x) ) + R^* (w).$$

\end{pro}

\begin{proof}

  As $R^* (w)= V^*(\sigma (w) ) - V^*(w) - A^*(w)$, we get $$V^*(w) - V^*(\sigma (w) ) + R^* (w) =  - A^*(w),$$ and, now using $x= f(\psi_{w}x)$, we get $$V(x) - V(\psi_{w}x) = V(f(\psi_{w}x)) - V(\psi_{w}x) - A(\psi_{w}x) + A(\psi_{w}x) = R(\psi_{w}x) + A(\psi_{w}x).$$

Substituting the above in the previous equation we get
\begin{align*}
(V^* + V - W)(x,w) - (V^* + V - W)(\psi_{w}x , \sigma(w)) + R^* (w) & = \\
[V^*(w) - V^*(\sigma (w) ) + R^* (w)] + [V(x) - V(\psi_{w}x)]  - W(x,w) + W (\psi_{w}x , \sigma(w))  & =\\
- A^*(w) + R(\psi_{w}x) + A(\psi_{w}x) + W (\psi_{w}x , \sigma(w)) - W(x,w) & = \\
R(\psi_{w}x),&  \\
\end{align*}
because $A^*(w)= A(\psi_{w}x) + W (\psi_{w}x , \sigma(w)) - W(x,w)$. So the claim follows.

\end{proof}
\bigskip

Note that $R\geq 0$, because $V$ is a calibrated subaction.

Note also that given $w=(w_0,w_1,..),$ then, $\psi_{w}(x)$ depends only of $w_0$. We can use either notation $\psi_{w}(x)$, or $\psi_{w_0}(x)$.

\bigskip

We know that the calibrated subaction satisfies
$$V(x)= \max_{w \in \Sigma} (-V^* -I^* +  W)(w,x).$$
Then, we define
$$b(w,x)= (V^* + V  +I^* - W)(w,x) \geq 0,$$ and,
$$\Gamma_{V}=\{(w,x) \in \Sigma \times [0,1] \,|\, V(x)=  (-V^* -I^* +  W)(w,x)\},$$
which can be written in an equivalent form
$$\Gamma_{V}=\{(w,x) \in \Sigma \times [0,1] \,| \, b(w,x)=0\}.$$

\begin{rem}
Note, that $b(w,x)=0$, if and only if, $(w,x)$ is an
optimal pair.
\end{rem}

Using $ R^* (w)= I^*(w) - I^*(\sigma (w))$  (it follows from (\ref{est})), the FR becomes
$$R(\psi_{w}x) = (V^* + V - W)(w,x) - (V^* + V - W)( \sigma(w), \psi_{w}(x)) + I^*(w) - I^*(\sigma (w)),$$
or
$$R(\psi_{w}x) = b(w,x) - b( \sigma(w), \psi_{w}(x))\,\,\,\,\,(\text{FR1}).$$

From this main equation we get:\\

\begin{lemma}
If $\,\T^{-1}(w,x) = ( \sigma(w), \psi_{w}(x))$, then

a) $b - b \circ \T^{-1} (w,x) = R(\psi_{w}x)$;

b) The function $b$ it is  non-decreasing in the trajectories of
$\T$;

c) $\Gamma_{V}$ is backward invariant;

d) when $(w,x)$ is optimal then $R(\psi_w(x))=0$.
\end{lemma}

\textbf{Proof:} see \cite{CLO}.

\bigskip
In this way $\T^{-n}$ spread optimal pairs.
\bigskip

As $R\geq 0$, then the function $b$ is a kind of Lyapunov function  for the iteration
of $\T^{-1}$.

\bigskip
From now on we assume $d=2$.

\bigskip

It is known that if $A$ is twist, then  $x \to w_x$ (can be multi-valuated)   is
monotonous non-increasing (see \cite{Ba} \cite{LMST} \cite{CLO}).  We recall the proof:

\vskip 0.5cm
\begin{pro} If $A$ is twist, then
$x \to w_x$ is
monotonous non-increasing.

\end{pro}

\begin{proof}
 Suppose $x<x'$, and, that $(w_x,x), ( w_{x'},x')$ are two optimal pairs. We will show that $w_x\geq w_{x'}.$

Indeed, as

$$V(x) =
\sup_{ w \in \Sigma}\, (W(w ,x) - V^* (w )- I^*(w))= W(w_{x},x) -
V^* (w_{x})- I^*(w_{x}),
$$
then
$$W(w ,x) - V^* (w )- I^*(w) \leq W(w_{x},x) - V^* (w_{x})- I^*(w_{x}),\,\,\,(*)$$
for any $w$, and we also have that

$$V(x') =
\sup_{ w \in \Sigma}\, (W(w ,x') - V^* (w )- I^*(w))= W(w_{x'},x') - V^* (w_{x'})- I^*(w_{x'}).
$$

Therefore,
$$W(w ,x') - V^* (w )- I^*(w) \leq W(w_{x'},x') - V^* (w_{x'})- I^*(w_{x'}),\,\,\,(**)$$
for any $w$.

Suppose, $x < x'$. Substituting $w_{x'}$ in the first expression (*), and $w_{x}$ in the second one (**) we get
$$\Delta(x,x',w_{x'}) \leq  \Delta(x,x',w_{x}),$$
where $W(x,w)-W(x',w)=\Delta(x,x',w)$. So the twist property implies that $w_{x'} \leq w_{x}$.

\end{proof}

We showed before that  the twist property implies that for $x<x'$,
if $b(w,x)=0$ and $b(w',x')=0$, then $w' < w$, which means that the
optimal sequences are monotonous  non-increasing. Remember,  that we define
the ``turning point $c$"  as being the maximum of the point
$x$ that has his optimal sequence starting in 1:
$$c=\sup\{x\, |\, b(w,x)=0 \,\Rightarrow \, w=(1\, w_1\, w_2 ...)\}.$$


The main criteria is the following:\\

\emph{``If $x \in [0,1]$ has the optimal sequence  $w=(w_0 \, w_1 \,
w_2\, ...)$ then
\[
w_0= \left \{
  \begin{array}{ll}
    1, \,\, if &  x \in [0,c] \\
    0, \,\, if &  x \in (c,1]
  \end{array}
\right.
\]}
Starting from $(x^0,w^0)$ we can iterate FR1 by $\T^{-n} (w,x)
=(w^n,x^n)$ in order to obtain new points $w^1 , \, w^2\, ...\in
\Sigma$. Unless the only possible optimal point $w(x)$, for all $x$,
is a fixed point for $\sigma$, then,  $0<c<1$.

Note that for $c$ there are two optimal pairs $(w,c)$ and $(w',c)$,
where the first symbol of $w$ is zero, and, the first symbol of $w'$
is one.

 The next lemma shows an interesting property of optimal pairs. If the maximizing measure for $A$ is supported in a periodic orbit, then the optimal pair $(w_p,p)$, for such points $p$ in the periodic orbit, could not be unique (that is, there exists more the one $w_p$ for a fixed $p$). This can happen (and there examples) in the case the turning point $c$ belongs to the pre-image of the maximizing periodic orbit.

\begin{lemma}
If $A$ satisfies the twist property, then $c$ is solution of
$$V(\psi_{1}x) + A(\psi_{1}x) = V(\psi_{0}x) + A(\psi_{0}x).$$
\end{lemma}
\begin{proof} As  for $y<c$, we have $b(w=(1 \, ...),y)=0$,
 taking limit of $y$ on the left side of $c$, then, we have from FR1, that $R(\psi_{1}y)=0$. From this follows  $R(\psi_{1}c)=0$, which means $V(c)=V(\psi_{1}c) + A(\psi_{1}c)$.  Analogously, taking limit of $y$ on the right of $c$, we get $V(c)=V(\psi_{0}c) + A(\psi_{0}c)$.  Thus, $V(\psi_{1}c) + A(\psi_{1}c)= V(\psi_{1}c) + A(\psi_{1}c)$.
\end{proof}

A point $x$ is called eventually periodic (or, pre-periodic), if
there is $n\neq m $, such that, $f^n(x)=f^m(x).$

\begin{lemma} \label{ufa} (Characterization of optimal change)
Let $c \in (0,1)$ be the turning point then, for any $x < x'$, such
that, $b(w,x)=0$ and $b(w',x')=0$, we have $w \neq w'$,  if, and
only if, there exists $n \geq 0$ such that $f^n (c) \in [x,x']$.
Moreover, if $x,x'$ are such that $w(x)$ and $w(x')$ are identical
until the $n$ coordinate, then, $f^n (c)\in (x,x').$
\end{lemma}
\begin{proof}

$^{ }$\\

Step 0\\
If $x < x' \leq c$, then, $w_{0} = w'_{0}=1$, else if $c< x < x' $,
then $w_{0} = w'_{0}=0$. Suppose $w_{0} = w'_{0}= i\in \{0,1\}$ then
applying FR1 we get  $\psi_{i}x < \psi_{i}x'$ and $b(
(w_{1} \, w_{2} \, ...), \psi_{i}x)=0$ and $b( (w'_{1} \, w'_{2} \,
...),\psi_{i}x' )=0$.

Step 1\\
If $\psi_{i}x < \psi_{i}x' \leq c$, then, $w_{1} = w'_{1}=1$, else,
if $c< \psi_{i}x < \psi_{i}x' $, then $w_{1} = w'_{1}=0$. Otherwise,
if $\psi_{1}x <  c < \psi_{1}x' $ we can use the monotonicity of $f$
in each branch in order to get $ x <  f(c) < x' $. Thus $$w_{1} \neq
w'_{1} \Leftrightarrow  x <  f(c) < x'. $$ The conclusion comes by
iterating this algorithm.

\end{proof}
\bigskip

\begin{lemma}

The set
$$B(w)=\{ x \, | \, b(w,x)=0 \}$$
is closed and connected, that is, an interval (could be a single
point). More specifically, if $B(w)=[a,b]$, then, $a$ and $b$ are
adherence points of the orbit of $c$.

In particular, if $c$ is pre-periodic, then, for any non-empty
$B(w)$, there exists $n,m$ such that $B(w)=[f^n (c), f^m (c)]$
(unless $B(w)$ is of the form $[0,b]$, or $[a,1]$.
\end{lemma}
\begin{proof}
 Indeed, remember that $\psi_{i}$, for $i=0,1$, are
order preserving. If, $x<y$, and, $x,y \in B(w)$, then, we claim
that each $z \in (x,y)$ satisfies $z \in B(w)$. Indeed, otherwise if
$\tilde{w} \neq w$ is the optimal sequence for $z$, we know that
there is $K>0$ such that $\tilde{w}_{j} = w_{j} $ for $j=0,..,k-1$
and $\tilde{w}_{k} \neq w_{k}$. On the other hand
$$\psi_{k, w} x < \psi_{k, \tilde{w}} z <\psi_{k, w} y.$$
Without lost of generality suppose $w_{k}=1$ then $\tilde{w}_{k}=0$
a contradiction by twist property, analogously if $w_{k}=0$ then
$\tilde{w}_{k}=1$ a contradiction again.

The closeness follows from the continuity on $x$ of the function
$b$: if, $x_n \subset B(w)$, and $x_n \to \bar{x}$, we observe that
$$b(w,x_n)=0  \Leftrightarrow  V(x_n)+ V^*(w) +I^*(w) - W(w,x_n)=0,$$
and this implies $b(w,\bar{x})=0$, that is, $\bar{x} \in B(w)$.

For the second part it is enough to see that, for each extreme of
the interval, for example $b$, if the optimal $w$ is not constant in
the right side, for any $\delta >0$ there is a image of $c$, namely
$f^{j}(c)\in [b, b+\delta)$ taking $\delta \to 0$ we get $f^{j}(c)
\to b^{+}$.

\end{proof}

\begin{rem}
Each set $B(w)=[a,b]$ is such that $a=f^n(c)$, or,
$a$ it is accumulated by a subsequence of $f^j(c)$ from the left
side. Similar property is true for $b$ (accumulated by the right
side).

\end{rem}

\bigskip

\begin{center}
\includegraphics[scale=0.25,angle=0]{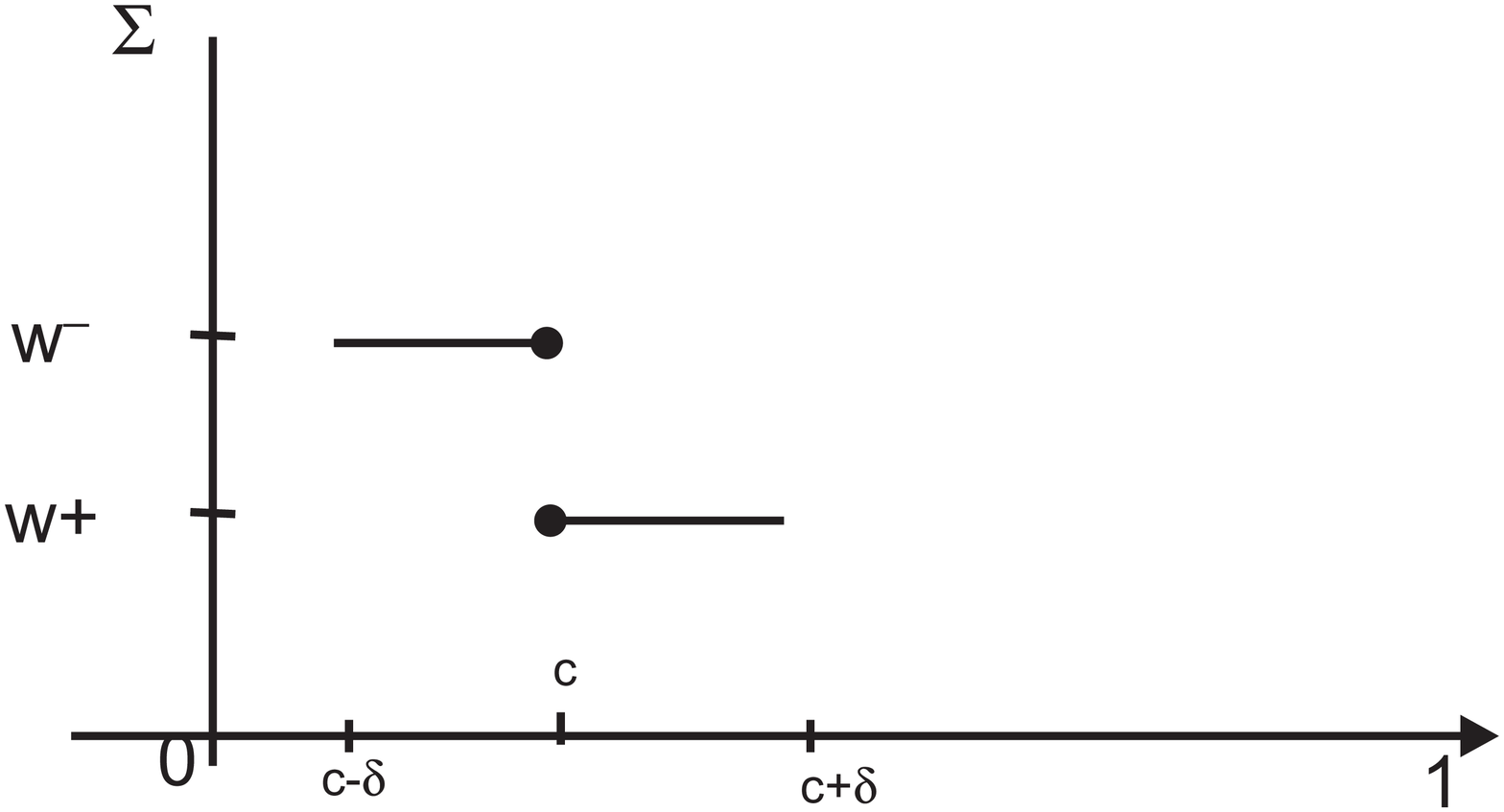}\\
\small{Fig. 3}  \\
\end{center}

\bigskip

\begin{lemma}\label{iso}
Let $c \in (0,1)$ be the turning point.  Let us suppose the $c$ is
isolated from his orbit, which means that, there is $\delta>0$, $w^{-}$, $w^{+}$, such
that, $b(w^{-},x)=0$, for any $x \in (c-\delta, c]$, and,
$b(w^{+},x)=0$, for any $x \in [ c, c+\delta)$,  then, there is no
accumulation points of the orbit of $c$. In this case $c$ is
pre-periodic.
\end{lemma}

\begin{proof}

Take $N>0$, such that $\frac{1}{2^{N-1}}< \delta$, and,
consider the sequence
$$\{c, f(c), ..., f^{N-1}(c)\},$$
which gives an
partition, which will be denoted by: $\{I_0, I_1, ..., I_{N-1} \}$.
 Note that the points $f^j(c), j=0,...,N-1$, are not order by $j$. A typical interval would be of the form $I_k=(f^{j_k}, f^{j_{k+1}})$. One of the $I_j$ contains the point $0$ in the boundary, and one contains the point $1$ in the boundary. It may happen that a certain $f^r (c)\in I_j$, but, then $r>N-1$.

Since each interval $I_J$ does not have in its interior points of
the form $f^k(c)$, $k\leq N-1$, we  get from Lemma \ref{ufa} above that:
$$ b(w,x)=0  \to w \in \overline{ i_0 i_1 ... i_{N-1}}, \forall x \in I_{j},$$
where $\overline{ i_0 i_1 ... i_{N-1}}\subset \Sigma$ denotes the cylinder with the corresponding symbols. That is, the discrepancy of the corresponding $w$ have to be at order bigger than $N-1$.

\bigskip

If $c$ is eventually periodic there exist just a finite number
intervals $B(w)$ with positive length. The other $B(w)$ are reduced
to points and they are also finite.
\bigskip

On the other hand, we claim that $I_{j}=[a,b]$ can have in its interior at most one in the forward orbit of $c$.

Indeed, if $I_{j} \cap \{f^{N}(c), f^{N+1}(c), ...\} = \emptyset$,
then the optimal $w$ will be constant and $I_{j}$ of the form $[f^n
(c), f^m (c)]$. Else, if  $f^k (c) \in I_{j} \cap \{f^{N}(c),
f^{N+1}(c), ...\} \neq \emptyset$, for $ k \geq N$, we denote by $k$
the minimum one where this happens.  Then,  we get
$$ b(w,x)=0  \to w \in \overline{ i_0 i_1 ... i_{k-1}}, \forall x \in I_{j}.$$

\bigskip

\begin{center}
\includegraphics[scale=0.25,angle=0]{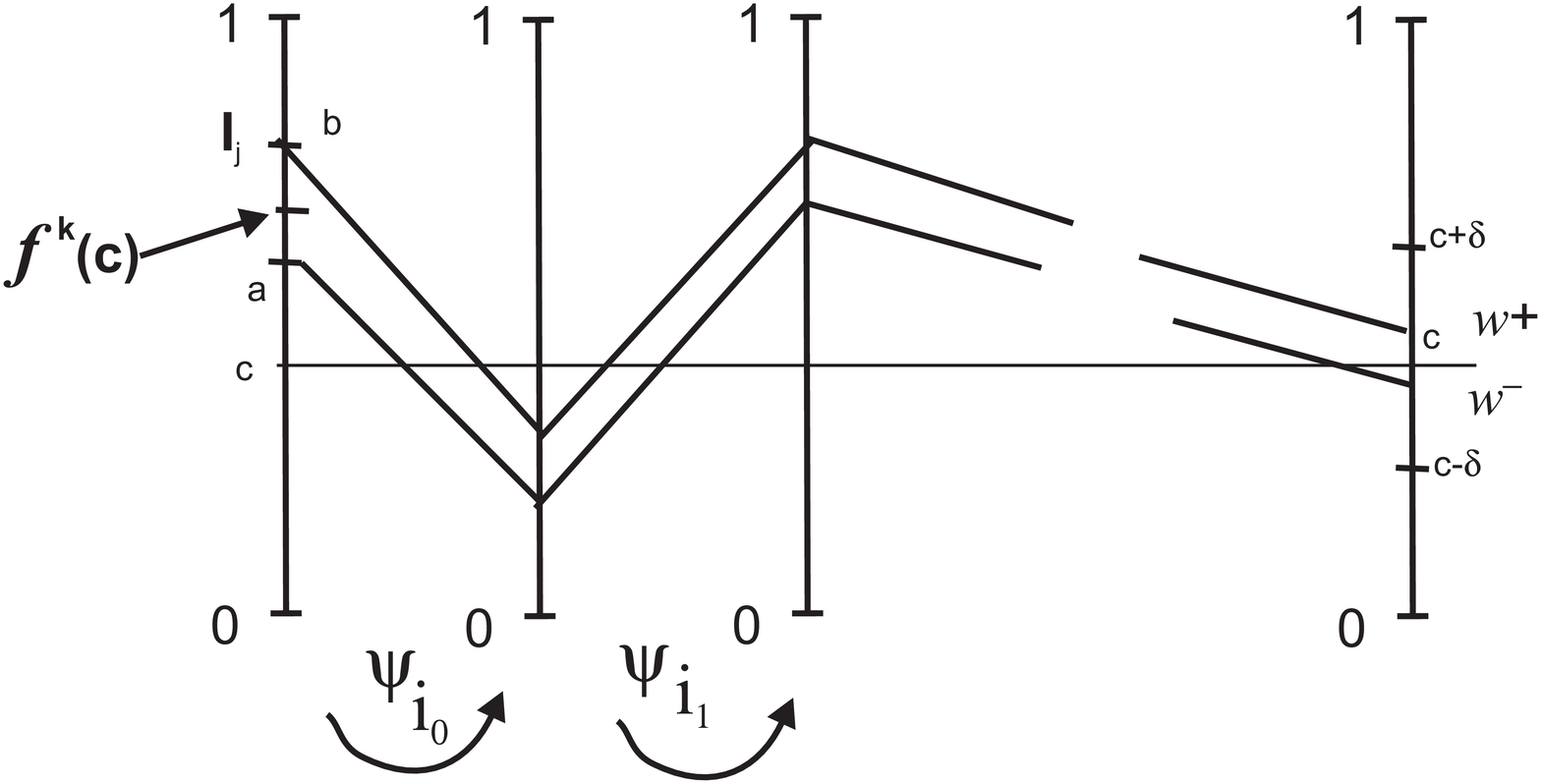}\\
\small{Fig.  4}  \\
\end{center}

\bigskip

If, we iterate the $k-1$ times the FR1, then $c \in
Z_{j}=\psi_{i_{k-1}}...\psi_{i_{0}} I_{j}$. By the choice of $N$ we
get $Z_{j} \subset (c-\delta, c+\delta)$ (see Fig. 4). Dividing
$I_{j}=[a,f^k (c)] \cup [f^k (c), b]$ we get
$$ b(w,x)=0  \to w =(i_0 i_1 ... i_{k-1} * w^-), \forall x \in [a,f^k
(c)]
,$$
and,
$$ b(w,x)=0  \to w =(i_0 i_1 ... i_{k-1} * w^+), \forall x \in [f^k (c), b].$$

Therefore, there is no room for another $f^r(c)$, $r\neq k$, to
belong to $I_j$.

\end{proof}

\begin{rem}
The main problem we have to face is the possibility
that the orbit of $c$ is dense in $[0,1]$.

\end{rem}

In the case $f$ is $d$ to one, we have to consider a finite number
of turning points, and, similar results can also be obtained.

\section{The countable and the good conditions}

We can see from last section that  the subaction $V$ will be
analytic, up to a finite number of points, if and only if,  the
point $c$ is eventually periodic. We would like to have sufficient
conditions for this happen.

We point out that if the maximizing probability for $A$ is a periodic orbit, then, the same happen for $A^*$ (see \cite{LOT} \cite{BLT}).

Remember  that a necessary condition for $w$ to be optimal for a
some $x$ is that $I^*(w)<\infty$.

In \cite{LMST} proposition 19 page 40 it is shown that, if $I^*(w) $ is finite,
then

$$ \lim_{n \to \infty} \frac{1}{n} \sum_{j=0}^{n-1}\delta_{\sigma^j(w)}= \mu_{\infty\, A}^* .$$

In principle, it can
exist an uncountable number of points $w$ such that the above limit can occur.

\begin{defi} We say a continuous $A: [0,1] \to \mathbb{R}$   satisfies the
the countable  condition, if there are a countable number of
possible optimal $w(x)$, when $x$ ranges over the interval $[0,1]$.

\end{defi}

We denote by $M$ the support of the maximizing probability periodic
orbit for $A^*$.

Consider the compact set of points $P=\{w\in \Sigma$, such that
$\sigma(w)\in M$, and $w$ is not on $M\}$.
\bigskip

 \begin{defi}
 We say that $A$ is good, if  for each $w\in P$, we have that $R^*(w)>0$, where $A^*$ is a dual of $A$.
\end{defi}

If $A$ is good, according to \cite{CLO}, a point $w$ satisfies
$I^*(w)<\infty$, if, an only if, $w$ is in the pre-image of the
maximizing periodic orbit. Such set of $w$ is countable, therefore,
if $A$ is good, then $A$ satisfies the countable condition. The good
condition, in principle, is more easy to be checked.

\begin{lemma} Suppose $A$ satisfies the twist and the countable condition.
Then there is at least one $B(w)$ with positive length of the form
$(f^n (c), f^m (c))$. Moreover, for  any subinterval $(a,b)$ there
exists at least one $B(w)$ with positive length of the form $(f^n
(c), f^m (c))$ inside $(a,b)$.

\end{lemma}

\begin{proof}

Denote the possible $w$, such that, $I^*(w)< \infty$, by $w_j$, $j
\in \mathbb{N}$.

For each $w^j$, $j \in \mathbb{N}$, denote $I_j=B(w^j)$, the maximal
interval where for all $x \in I_j$, we have that, $(x,w^j)$ is an
optimal pair. Some of these intervals could be eventually a point,
but, an infinite number of them have positive length, because the
set $[0,1]$ is not countable. We consider from now on just the ones
with positive length.

Note that by the same reason, in each subinterval $(e,u)$, there
exists an infinite countable number of $B(w)$ with positive length.

We suppose, by contradiction, that each interval $B(w)=[a,b]$, with
positive length is such that, each side is approximated by a
sub-sequence of points $f^j(c)$.

Take one interval $(a_1,b_1)$ with positive length inside $(0,1)$.
There is another one $(a_2,b_2)$  inside  $(0,a_1)$, and one more
$(a_3,b_3)$ inside $(b_1,1)$.

If we remove from the interval $[0,1]$ these three intervals we get
four intervals. Using our hypothesis, we can find new intervals with
positive length inside each one of them. Then we do the same removal
procedure as before. This procedure is similar to the construction
of the Cantor set. If we proceed inductively on this way, the set of
points $x$ which remains after infinite steps is not countable. An
uncountable number of such $x$ has a different $w(x)$. This is not
possible because the optimal $w(x)$ are countable.

Then, the first claim of the lemma is true.

Given an interval $(a,b) \subset (0,1)$, we can do the same and use
the fact that $(a,b)$ is not countable.

\end{proof}

\bigskip

\begin{lemma} \label{both} Suppose $A$ satisfies the twist and the countable condition. If $c$ is the turning point, then, there is $\delta>0$, $w^{-},
w^{+}$, such that, $b(w^{-},x)=0$, for any $x \in (c-\delta, c]$,
{\bf and}, $b(w^{+},x)=0$, for any $x \in [ c, c+\delta)$.

That is, $c$ is isolated of its forward orbit by both sides.
\end{lemma}

\begin{proof}

If there exist just a finite number of intervals, then $c$ is
eventually periodic. We will suppose $c$ is not eventually periodic,
and, we will reach a contradiction. Therefore, if $f^n(c) = f^m
(c),$ then $m=n$.

Denote by $I_j=[a_j,b_j]$. We denote $I_0$ the interval of the form
$[0,b_0]$, and,  $I_1$ the interval of the form $[a_0,1]$. From last
lemma, there is $j\neq 0,1$, and $n_j$ and $m_j$, such that $a_j=
f^{n_j}(c)$ and $b_j= f^{m_j}(c)$.

Suppose first that  $n_j<m_j$.

Consider the inverse branch $\psi_{i_1^j}$, where $i_1^j=i_1$ is
such that $\psi_{i_1}((f^{n_j})(c) )= f^{n_j-1}(c)$. This $i_1$
do not have to be the first symbol of the optimal $w$  for
$f^{n_j}(c)$. Then, $\psi_{i_1}(I_j)$ is another interval, which is
strictly inside a domain of injectivity of $f$, does not contain any
forward image of $c$, and in its left side we have the point
$f^{n_j-1}(c)$.

 Then, repeating the same procedure
inductively, we get $i_2$, such that
$$\psi_{i_2}((f^{n_j-1})(c) )=
f^{n_j-2}(c),$$
determining another interval which does not contain
any forward image of $c$, and in his left side we have the point
$f^{n_j-2}(c)$. Repeating the reasoning over and over again, always
taking the same inverse branch which contain $f^n(c)$, $0\leq n \leq
n_j$, after $n_j$ times we arrive in an interval of the form $(c,
r_j)$. Note that each inverse branch preserves order. It is not
possible to have an iterate $f^k(c)$, $k \in \mathbb{N}$, inside
this interval $(c, r_j)$ (by the definition of $I_j$). Then, the
optimal $w$ for $x$ in this interval $(c, r_j)$ is a certain
$\tilde{w_j}$ which can be different of $\sigma^{n_j} (w_j).$

Suppose now that  $n_j>m_j$.

 Using the analogous procedure we get that there exists $r^j$,
such that the optimal $w(x)$ for $x$ in the interval $(r^j, c)$ is a
certain $\hat{w_j}.$

If both cases happen, then $c$ is eventually periodic.

The trouble happens when just one type of inequality is true.
Suppose without lost of generality that we have always $n_j<m_j$.

Let's fix for good a certain $j$.

Therefore, all we can get with the above procedure is that  $c$ is
isolated by the right side

In the procedure of taking pre-image of $f^{n_j}(c)$, always
following the forward orbit $f^n(x)$, $0\leq n \leq n_j$, we will
get a sequence of $i_1,i_2,...,i_{n_j}$.  In the first step  we have
two possibilities: $\psi_{i_1} (f^{m-j}(c)) = f^{m_j-1}(c)$, or not.

If it happens the second case, we are done.  Indeed, the interval
$\psi_{i_1} [f^{n_j}(c),f^{m_j}(c))]$ does not contain forward
images of $c$ (otherwise $[f^{n_j}(c)),f^{m_j}(c))]$ would also
have). Now we follow the same procedure as before, but, this time
following the branches which contains the orbit of $f^m (c)$, $0\leq
m\leq m_j$. In this way, we get that $c$ is isolated by the left
side.

Suppose  $\psi_{i_1} (f^{m_j}(c))  = f^{m_j-1}(c)$. Consider the
interval, $[f^{n_j-1}(c)),f^{m_j-1}(c))]$, which do not contain
forward images of $c$.

Now, you can ask the same question:  $\psi_{i_2} (f^{m_j-1}(c)) =
f^{m_j-2}(c)$? If this do not happen (called the second option),
then, in the same way as before, we are done ($c$ is also isolated
by the right side). If the expression is true, then, we proceed with
the same reasoning as before.

We proceed in an inductive way until time $n_j$. If in some time we
have the second option, we are done, otherwise, we show that any
$x\in (c, f^{m_j-n_j}(c))$ has a unique optimal $w(x)$ (there is no
forward image of $c$ inside it).

Denote $k=m_j-n_j$ for the $j$ we fixed.

From the above we have that for any $B(w)$, which is an interval of
the form $[f^{n_i}(c)),f^{m_i}(c))]$, for any possible $i$, it is
true that $m_i-n_i=k.$

We claim that the set of points $x$ which are extreme  points of any
$B(\tilde{w})$,  and, such that $x$ can be approximated by the
forward orbit of $c$ is finite. Suppose without lost of generality
that $x$ is the right point of a $B(w)=(z,x)$.

If the above happens, then, by the last lemma, applied to
$(a,b)=(x,\epsilon)$, $\epsilon$ small, we have an infinite sequence
of intervals of the form $[f^{n_i}(c)),f^{n_i+k}(c))]$, such that
$f^{n_i}(c)\to x$, as $n_i\to \infty.$ Therefore, $x$ is a periodic
point of period $k$. There are a finite number of points of period
$k$. This shows our main claim. Finally, $c$ is eventually periodic.

\end{proof}

\begin{thm} \label{maincor1} Suppose $A$ satisfies the twist and the countable condition, that the maximizing probability is unique, and, also that it is a
periodic orbit, then $V$ is analytic, up to a finite number of
points.

\end{thm}

\begin{proof}

It follows from Lemma \ref{both} and Lemma \ref{iso}.

\end{proof}

Therefore, we get:

\begin{thm} \label{maincor} Suppose $A$ satisfies the twist condition and that the turning point is eventually periodic,
that the maximizing probability is unique, and, also that it is a
periodic orbit, then $V$ is analytic, up to a finite number of
points.

\end{thm}

The next example shows that the theory we just presented above allows
one to compute, via an algorithm, the calibrated sub-action $V$.
By this, we mean that, if we know $W$, and, we have some information
about the combinatorics of the position of the maximizing orbit, then, we can get the subaction $V$.

\begin{example}\label{ex1}

We assume that $m(A)=0$ and $T(x)=2\,x$ (mod 1).

In this example we consider an measure supported in a periodic orbit of period 4, $x_{0}=\frac{1}{15}$ it is easy to see that in this case the optimal measure on $\hat{\Sigma}$ is supported in:
$$(\frac{1}{15}, \overline{1000}...), \; (\frac{2}{15}, \overline{0100}...), \; (\frac{4}{15}, \overline{0010}...),\text{ and }(\frac{8}{15}, \overline{0001}...).$$

By definition the turning point $c$ should be between $\frac{1}{15}$ and $\frac{2}{15}$. Let us consider two cases:\\

\textbf{Case 1: } Suppose that $c=\frac{2}{16}$, that is a eventually periodic point, but more than that, $c$ is a pre-image of order 4 of the fix point 1.

The orbit of $c$, which is given by $c=\frac{2}{16}, \, f(c)=\frac{4}{16}, \,  f^{2}(c)=\frac{8}{16}, \, f^{3}(c)=1$. In this way the $w(x)$ of the optimal pairs should be constant in the intervals:\\
$$I_1 =[0,c], \, I_2 =[c,f(c)], \, I_3 =[f(c),f^{2}(c)], \,  I_4 =[f^{2}(c),f^{3}(c)=1].$$

In this case the $c$ is pre-periodic.

 Note that the $f^k(c)$ are monotonous in $k$. This is not always the case for other examples.

Since there is one only periodic point in each interval we get:\\
$ b(x, \overline{1000}...)=0, \; \forall x \in I_1$\\
$ b(x, \overline{0100}...)=0, \; \forall x \in I_2$\\
$ b(x, \overline{0010}...)=0, \; \forall x \in I_3$\\
$ b(x, \overline{0001}...)=0, \; \forall x \in I_4$.\\

\begin{center}
\includegraphics[scale=0.3,angle=0]{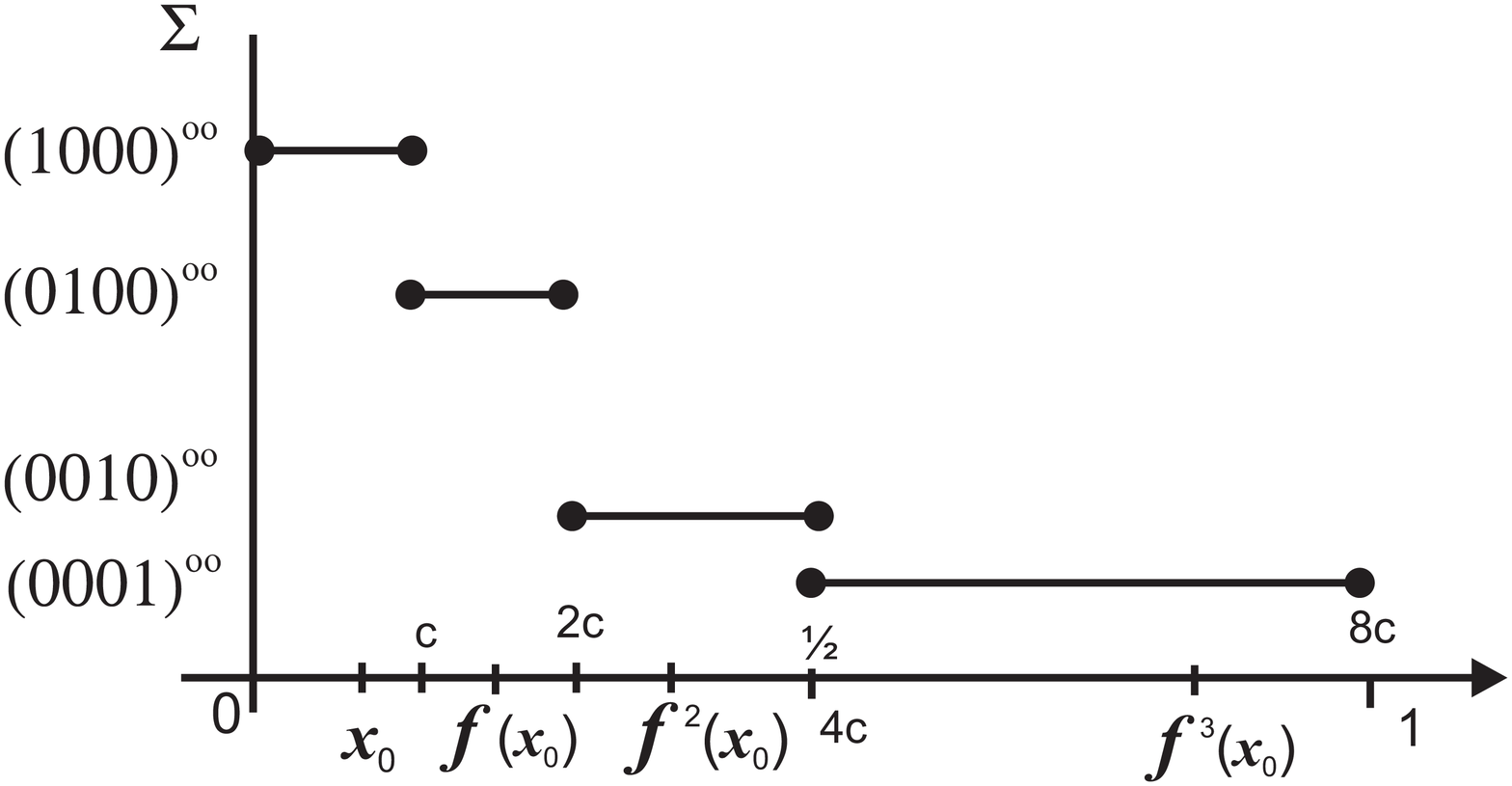}\\
\small{Fig. 5}  \\
\end{center}

Using the definition of $b$ we get\\
\[
V(x)= \left\{
  \begin{array}{ll}
    -V^*(\overline{1000}...) - 0 +W(x,\overline{1000}...) , & \forall x \in I_1\\
    -V^*(\overline{0100}...) - 0 +W(x,\overline{0100}...) , & \forall x \in I_2\\
    -V^*(\overline{0010}...) - 0 +W(x,\overline{0010}...) , & \forall x \in I_3\\
    -V^*(\overline{0001}...) - 0 +W(x,\overline{0001}...) , & \forall x \in I_4\\
  \end{array}
\right.
\]

The fundamental relation allow us to write:
\[
V(x)= \left\{
  \begin{array}{ll}
    V(\psi_{\overline{1000}...}x) + A(\psi_{\overline{1000}...}x), & \forall x \in I_1\\
    V(\psi_{\overline{0100}...}x) + A(\psi_{\overline{0100}...}x) , & \forall x \in I_2\\
    V(\psi_{\overline{0010}...}x) + A(\psi_{\overline{0010}...}x), & \forall x \in I_3\\
    V(\psi_{\overline{0001}...}x) + A(\psi_{\overline{0001}...}x), & \forall x \in I_4\\
  \end{array}
\right.
\]

In particular, applying in $x_0$ we have:\\
$$V(x_0)= V(\psi_{\overline{1000}...}x_0) + A(\psi_{\overline{1000}...}x_0)= V(f^{3}x_0) + A(f^{3}x_0)$$
$$V(f(x_0))= V(\psi_{\overline{0100}...}f(x_0)) + A(\psi_{\overline{0100}...}f(x_0))= V(x_0) + A(x_0)$$
$$V(f^{2}(x_0))= V(\psi_{\overline{0010}...}f^{2}(x_0)) + A(\psi_{\overline{0010}...}f^{2}(x_0))= V(f (x_0)) + A(f (x_0))$$
$$V(f^{3}(x_0))= V(\psi_{\overline{0001}...}f^{3}(x_0)) + A(\psi_{\overline{0001}...}f^{3}(x_0))= V(f^{2}(x_0)) + A(f^{2} (x_0))$$

Since $V$ is unique up to constants we can choose $V(x_0)=0$ for instance an solve the system finding:
$$V(x_0)=0, \; V(f(x_0))= A(x_0) ,
\;  $$ $$V(f^{2}x_0)=  A(x_0) -  A(f^{2} (x_0))\text{ and }V(f^{3}x_0)= - A(f^{3}x_0)$$

Using this results in the previous formula for $V$ we get the values of $V^*$:\\
1- $V(x_0)=0 = -V^*(\overline{1000}...) - 0 +W(x,\overline{1000}...)$, thus
 $$V^*(\overline{1000}...)=W(x_{0},\overline{1000}...).$$
2- $V(T(x_0))= A(x_0)= -V^*(\overline{0100}...) - 0 +W(f(x_0),\overline{0100}...)$, thus
$$V^*(\overline{0100}...)= -A(x_0) + W(f(x_0),\overline{0100}...).$$
3- $V(f^{2}(x_0))=  A(x_0) +  A(f (x_0))= -V^*(\overline{0010}...) - 0 +W(f^{2}(x_0),\overline{0010}...)$, thus
$$ V^*(\overline{0010}...) =- A(x_0) -   A(f (x_0)) +  W(f^{2}(x_0),\overline{0010}...)$$
4- $V(f^{3}(x_0))= - A(f^{3}x_0)= -V^*(\overline{0001}...) - 0 +W(f^{3}(x_0),\overline{0001}...)$, thus
$$V^*(\overline{0001}...) = A(f^{3}x_0) + W(f^{3}(x_0),\overline{0001}...)$$

So the explicit formula for $V$ depends on $A$ and is given by
\[
V(x)= \left\{
  \begin{array}{ll}
    W(x,\overline{1000}...) -W(x_{0},\overline{1000}...) , & \forall x \in I_1\\
    A(x_0) +W(x,\overline{0100}...) - W(f(x_0),\overline{0100}, & \forall x \in I_2\\
    A(x_0) + A(f(x_0))+W(x,\overline{0010}...)- W(f^{2}(x_0),\overline{0010}...) , & \forall x \in I_3\\
    -A(f^{3}x_0) + W(x,\overline{0001}...) - W(f^{3}(x_0),\overline{0001}...) , & \forall x \in I_4\\
  \end{array}
\right.
\]

Or
\[
V(x)= \left\{
  \begin{array}{ll}
    \Delta(x, x_{0}, \overline{1000}...), & \forall x \in I_1\\
    A(x_0) + \Delta(x, f(x_0),\overline{0100}...), & \forall x \in I_2\\
    A(x_0) +   A(f (x_0))   + \Delta(x, f^{2}(x_0),\overline{0010}...), & \forall x \in I_3\\
    -A(f^{3}x_0) +\Delta(x,f^{3}(x_0),\overline{0001}...), & \forall x \in I_4\\
  \end{array}
\right.
\]

\end{example}

\section{The optimal solution when the maximizing probability  is not a periodic orbit}
\bigskip

We are going to analyze now the  variation of the optimal point when the support of the maximizing probability is
not necessarily a periodic orbit. What can be said in the general case?

Consider the subaction defined by,
$$V(x) =
\sup_{ w \in \Sigma}\, (H_\infty (w,x)  -\, I^*(w))
$$

Remember that as $I^*$ is lower semicontinuous and
$H_\infty=W- V^*$ is continuous, then for each fixed $x$, the supremum of
$H_\infty (w,x)  -\, I^*(w)$ in the variable $w$ is achieved, and we
denote (one of such $w$) it by $w(x)$. In this case we say $w(x)$ is
an optimal point for $x$.

We want to show that $w(x)$ is  unique for the generic $x$

Define the multi-valuated function $U: [0,1] \to \Sigma$ given by:
$$U(x)=\{w(x) | x \in [0,1]\}$$

As graph $(U)$ is closed in each fiber, and $\Sigma $  is compact we can define:
$$u^{+}(x)=\max U(x), \text{ and } u^{-}(x)=\min U(x).$$

Since the potential $A$ is twist we know that $U$ is a monotone not-increasing multi-valuated function, that is,
$$u^{-}(x) \geq u^{+}(x+\delta),$$
when $x<x +\delta$. In particular are monotone not-increasing single-valuated functions.
\bigskip

\begin{center}
\includegraphics[scale=0.33,angle=0]{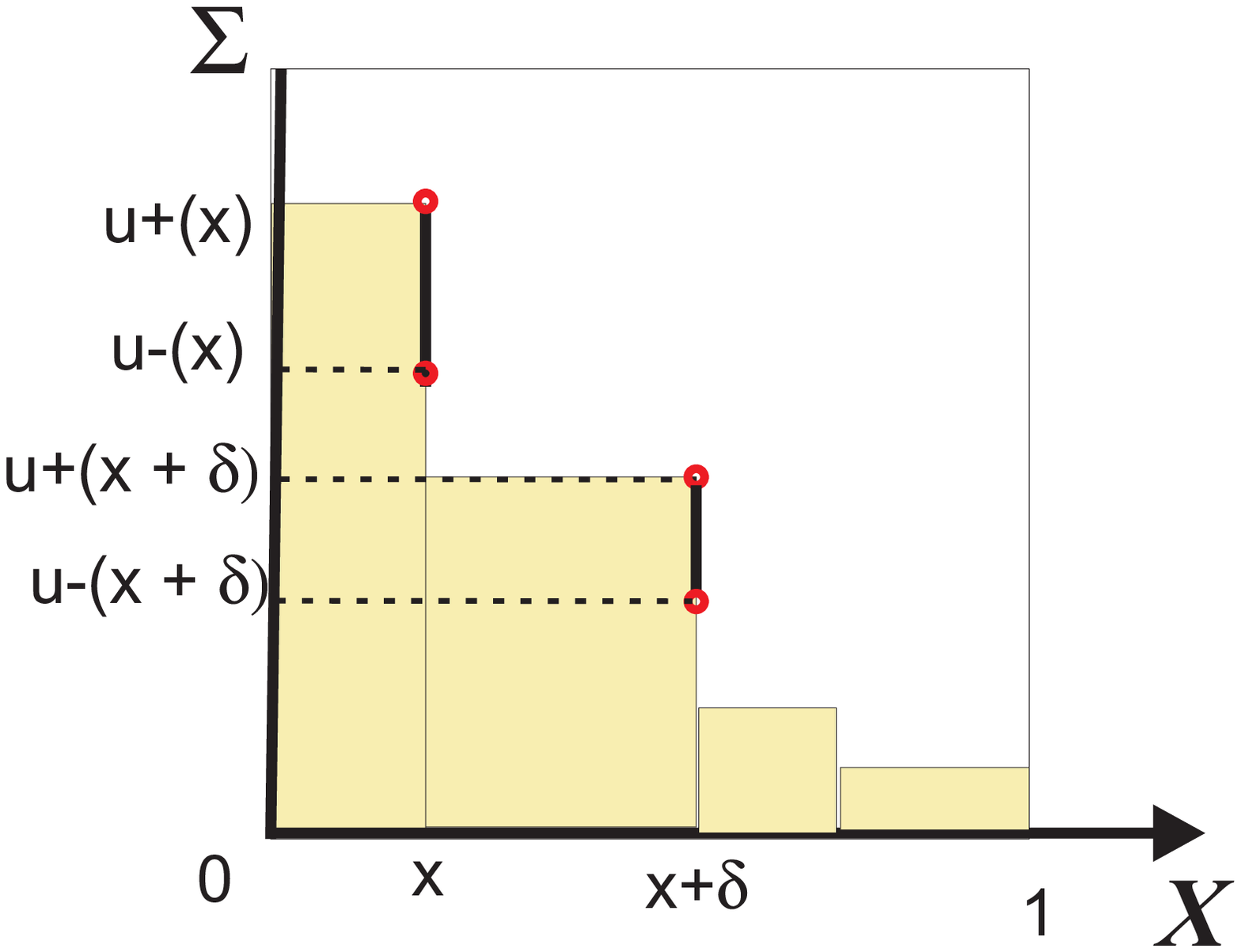}\\
\small{fig. 6 \,\, The graph of $U$}\\
\end{center}

\bigskip

We claim that $u^{+}$ is left continuous.
In order to conclude that, take a sequence $x_{n}\to x$ on the left side. Consider, the sequence $u^{+}(x_{n}) \in \Sigma$, so its set of accumulation points is contained in $U(x)$. Indeed, suppose $\liminf u^{+}(x_{n}) \to \tilde{w} \in \Sigma$. In one hand, we have, $V(x_{n}) = H_\infty
(u^{+}(x_{n}),x_{n})  -\, I^*(u^{+}(x_{n})).$ Taking limits on this equation and using the continuity of $V$ and $H_\infty$  and the lower semicontinuity of $I^*$ we get,
$$ V(x) \leq  H_\infty(\tilde{ w},x)  - \, I^*(\tilde{ w}).$$
Because $\liminf I^*(u^{+}(x_{n})) \geq I^*(\tilde{ w})$. So $\tilde{w} \in U(x)$. On the other hand, $u^{+}$ is monotone not-increasing, so
$u^{+}(x_{n}) \geq u^{+}(x)$. From the previous we get
$$\limsup u^{+}(x_{n}) \geq u^{+}(x) \geq  \tilde{w} =\liminf u^{+}(x_{n}),$$
that is,
$$\lim_{x_{n}\to x^{-}} u^{+}(x_{n}) = u^{+}(x).$$

Now consider a sequence $x_{n}\to x$ on the right side. Take, the sequence $u^{+}(x_{n}) \in \Sigma$, so its set of accumulation points is not necessarily contained in $U(x)$. However it is the case.
Let $x_{n_{k}}$ be a subsequence such that, $u^{+}(x_{n_{k}}) \to \tilde{w}$.

We know that $V(x_{n_{k}}) = H_\infty
(u^{+}(x_{n_{k}}),x_{n_{k}})  -\, I^*(u^{+}(x_{n_{k}})).$
Taking limits on this equation and using the uniform continuity of $V$ and $H_\infty$ we get
$$ I^*(\tilde{w}) \leq  \liminf_{k \to \infty} I^*(u^{+}(x_{n_{k}})) = $$ $$= \liminf_{k \to \infty} H_\infty
(u^{+}(x_{n_{k}}),x_{n_{k}})  -\,V(x_{n_{k}})= H_\infty
(\tilde{w},x)  -\,V(x).$$

In other words, $ V(x) \leq H_\infty
(\tilde{w},x)  -\, I^*(\tilde{w}),$ that is, $\tilde{w} \in U(x)$. So $$cl (u^{+}(x_{n})) \subseteq U(x).$$

Since $u^{+}$ is monotone not-increasing,
$u^{+}(x_{n}) \leq u^{+}(x)$, thus
 $$\limsup u^{+}(x_{n}) \leq u^{+}(x),$$
that is, $u^{+}$  is right upper-semicontinuous.

It is known that for  any USC function defined in a complete metric space the set of points of continuity is generic.

Therefore, we get that:

\begin{thm} \label{ElR}
For a  generic $x$ we have that $U(x)=\{ u^+(x) = u^-(x) \}$ and $w(x)$ is unique.
\end{thm}


 \begin{proof}

Indeed, suppose that there is a point in the set of continuity of $u^+(x)$ such that, $u^+(x) > u^-(x)$ so the monotonicity of $U$ implies that
$$u^+(x)> u^{-}(x) \geq u^{+}(x + \delta),$$
for all $\delta >0$. Contradicting the continuity.
\end{proof}

\end{document}